%% file: main.tex
\documentclass{mcom-l}

\usepackage{graphicx}
\usepackage[utf8]{inputenc}
\usepackage{subfigure}
\usepackage{amsmath}
\usepackage{amsthm}
\usepackage{amssymb, mathtools, hyperref}
\usepackage{graphicx}
\usepackage{cleveref}
\usepackage{tikz}
\usetikzlibrary{patterns,shapes}
\usepackage{gastex}
\usepackage{todonotes}
\usepackage{cite}

\usepackage{amsrefs}

\newcommand{\N}{\mathbb N}
\newcommand{\A}{\mathcal A}
\newcommand{\B}{\mathcal B}

\newcommand{\LL}{\mathcal L}

\newcommand{\R}{\mathcal R}

\newcommand{\uu}{\mathbf{u}}
\newcommand{\vv}{\mathbf{v}}

\newcommand{\yy}{\mathbf{y}}

\newcommand{\dd}{\mathbf{d}}

\newcommand{\CR}{\mathrm{E}}
\newcommand{\lcm}{\mathrm{lcm}}
\newcommand{\Per}{\mathrm{Per}}
\newcommand{\colour}{\mathrm{colour}}
\newcommand{\barva}{\mathrm{colour}}

\newcommand{\mat}[4]{\left(
	\begin{array}{cc}
		#1 & #2 \\
		#3 & #4 \\  
	\end{array}
	\right)
}

\theoremstyle{definition}
\newtheorem{definition}{Definition}
\newtheorem{corollary}[definition]{Corollary}
\newtheorem{remark}[definition]{Remark}
\newtheorem{example}[definition]{Example}

\theoremstyle{plain}
\newtheorem{theorem}[definition]{Theorem}
\newtheorem{proposition}[definition]{Proposition}
\newtheorem{lemma}[definition]{Lemma}
\newtheorem{observation}[definition]{Observation}

\begin{document}

% ------------------ Titulka pro AMS
\title{Asymptotic repetitive threshold of balanced sequences}

\author[L\!'. Dvo\v r\'akov\'a]{L\!'ubom\'ira Dvo\v r\'akov\'a}
\address{FNSPE Czech Technical University in Prague, Czech Republic}
\curraddr{}
\email{lubomira.dvorakova@fjfi.cvut.cz}
\thanks{The first and the third authors were supported by the Ministry of Education, Youth and Sports of the Czech Republic through the project  CZ.02.1.01/0.0/0.0/16\_019/0000778.}

%    author two information
\author[D. Opo\v censk\'a]{Daniela Opo\v censk\'a}
\address{FNSPE Czech Technical University in Prague, Czech Republic}
\curraddr{}
\email{opocedan@fjfi.cvut.cz}
\thanks{The second author was supported  by Czech technical university in Prague, through the project  SGS20/183/OHK4/3T/14.}

\author[E. Pelantov\'a]{Edita Pelantov\'a}
\address{FNSPE Czech Technical University in Prague, Czech Republic}
\curraddr{}
\email{edita.pelantova@fjfi.cvut.cz}
% \thanks{The first and the third authors were supported by the Ministry of Education, Youth and Sports of the Czech Republic through the project  CZ.02.1.01/0.0/0.0/16\_019/0000778.}

%    \subjclass is required.

\date{}

\dedicatory{}

\begin{abstract} The critical exponent  $E(\mathbf u)$ of an infinite sequence $\mathbf u$ over a finite alphabet expresses the maximal repetition of a factor in $\mathbf u$.   By the famous Dejean's theorem,  $E(\mathbf u) \geq 1+\frac1{d-1}$ for every $d$-ary sequence $\mathbf u$. We define the asymptotic critical exponent $E^*(\mathbf u)$  as the upper limit  of  the maximal repetition of factors of length $n$.
We show that for any $d>1$  there exists a $d$-ary sequence $\uu$ having $E^*(\mathbf u)$ arbitrarily  close to  $1$.   
Then we focus on the class of $d$-ary balanced sequences. In this class,  the  values $E^*(\mathbf u)$ are bounded from below by  a threshold strictly bigger than 1. We provide a method which enables us to find a $d$-ary balanced sequence with the least asymptotic critical exponent for $2\leq d\leq 10$.  

\end{abstract}
\maketitle

% ---------------------------------------

% --------------------------- Titulka pro Élsevier
% \begin{frontmatter}
% \title{Asymptotic repetitive threshold  of balanced sequences}

% \author[cvut]{L\!'ubom\'ira Dvo\v r\'akov\'a}
% \ead{lubomira.dvorakova@fjfi.cvut.cz}
% \author[cvut]{Daniela Opo\v censk\'a}
% \ead{opocedan@fjfi.cvut.cz}
% \author[cvut]{Edita Pelantov\'a}
% \ead{edita.pelantova@fjfi.cvut.cz}

% \address[cvut]{FNSPE Czech Technical University in Prague, Czech Republic}

% \begin{abstract} The critical exponent  $E(\uu)$ of an infinite sequence $\uu$ over a finite alphabet expresses the maximal repetition of a factor in $\uu$.   By the famous Dejean's theorem,  $E(\uu) \geq 1+\frac1{d-1}$ for every $d$-ary sequence $\uu$. We define the asymptotic critical exponent $E^*(\uu)$  as the upper limit  of  the maximal repetition of factors of length $n$. We show 
% that for any $d>1$  there exists a $d$-ary sequence $\uu$ having $E^*(\uu)$ arbitrarily  close to  $1$.   
%  Then we focus on the class of $d$-ary balanced sequences. In this class,  the  values $E^*(\uu)$  are  bounded from below by  a threshold strictly bigger than 1. We provide a method which enables us to find a $d$-ary balanced sequence with the least asymptotic critical exponent for $2\leq d\leq 10$.  

% \end{abstract}

% \begin{keyword}
% balanced sequence \sep critical exponent \sep repetition threshold \sep constant gap sequence \sep return word \sep bispecial factor \sep Sturmian sequence

% \MSC 68R15
% \end{keyword}
% \end{frontmatter}

% ----------------------------------------------------------
%\input{Version1.tex}

\section{Introduction}

The concatenation of $e\in \mathbb{N}$ copies of a non-empty word $u$ is usually  abbreviated  as   $u^e$. In 1972, Dejean  extended this exponential notation to rational exponents.
If $u$ is a non-empty word of length $\ell$ and $e$ is a positive rational number of the form $n/\ell$, then  $u^e$ denotes the prefix of length $n$  of the infinite periodic sequence $uuu\cdots = u^\omega$.
For instance, a Czech word $starosta$ (mayor) can be written in this formalism  as  $(staro)^{8/5}$.  The rational  exponent $e$  describes the repetition rate of  $u$ in the string $u^e$.

The critical exponent  $E(\uu )$ of an infinite sequence $\uu = u_0u_1u_2 \cdots$ captures the maximal possible repetition rate of factors occurring in  $\uu$, formally,
$$E(\uu) =\sup\{e \in \mathbb{Q}: \  u ^e \  \text{is a factor of   } \uu  \  \text{for a non-empty word} \  u\}\,.
$$
The number $E(\uu)$  can be rational or irrational.  Krieger and Shallit~\cite{KrSh07} have shown that any positive real number larger than one is a critical exponent of some sequence   over a finite alphabet.  If  the size of the alphabet is a fixed number $d \in \mathbb{N}, d\geq 2$, then the critical exponent of any  sequence $\uu$ over this alphabet cannot be smaller than a threshold larger than one. This bound is denoted in~\cite{CuRa11} as $RT(d)$  and called the repetitive threshold, i.e.,
$$ RT(d) = \inf\{ E(\uu) : \uu \text{\ is a sequence over a $d$-ary alphabet} \}\,.
$$

For instance, if  $\uu$ is a binary sequence over $\{0,1\}$, then  $00$ or $11$ or $0101$ appears in  $\uu$, and thus $E(\uu) \geq 2$. The existence of an infinite binary sequence $\uu$ with $E(\uu) = 2$ was demonstrated by Thue in 1912. The binary sequence for which $E(\uu) = 2$ is nowadays known as the Thue-Morse sequence (or the Prouhet-Thue-Morse sequence).
Therefore $RT(2)=2$. Dejean~\cite{Dej72} proved that $RT(3)=\frac{7}{4}$ and this bound is attained. In the same paper she also conjectured:
\begin{itemize}
    \item $RT(4)=7/5$;
    \item $RT(d)=1+\frac{1}{d-1}$ for $d\geq 5$.
\end{itemize}
The conjecture had been proved step by step by many people
\cite{Pan84c, Mou92, MoCu07, Car07, CuRa11, Rao11}.

\medskip
Recently, Rampersad, Shallit and Vandome asked the same question for $d$-ary balanced sequences.
Let us recall that a sequence over a finite alphabet is balanced if, for any two of
its factors $u$ and $v$ of the same length, the number of occurrences of each letter
in $u$ and $v$ differs by at most 1.
Binary balanced aperiodic sequences coincide with Sturmian sequences and the least critical exponent is equal to $2+\frac{1+\sqrt{5}}{2}$ and it is reached by the Fibonacci sequence~\cite{CaLu2000}.
Let us denote the repetitive threshold of balanced sequences over a $d$-ary alphabet by $RTB(d)$, i.e.,
$$RTB(d) = \inf\{E(\vv): \vv \text{ is a balanced sequence over a } d \text{-ary alphabet}\}\,. $$
The following results have been proved so far:
\begin{itemize}
    \item $RTB(3)=2+\frac{1}{\sqrt{2}}$ and $RTB(4)=1+\frac{1+\sqrt{5}}{4}$ \cite{RSV19};
    \item $RTB(d)=1+\frac{1}{d-3}$ for $5 \leq d\leq 10$ \cite{Bar20, BaSh19, DDP21};
    \item $RTB(d)=1+\frac{1}{d-2}$ for $d=11$ and all even numbers $d \geq 12$ \cite{DOPS2022}.
    \end{itemize}
It remains as an open problem to prove the conjecture $RTB(d)=1+\frac{1}{d-2}$ also for all odd numbers $d\geq 13$.

\medskip

In this paper we focus on the asymptotic critical exponent $E^*(\uu)$ of $\uu$. It is defined to be $+\infty$ if $E(\uu) = +\infty$ and
$$E^*(\uu) =\limsup_{n\to \infty}\{e \in \mathbb{Q}: \  u ^e \  \text{is a factor of  } \uu  \  \text{for some }  u \ \text{of length} \  n  \}\,,$$
 otherwise. Obviously, $E^*(\uu) \leq E(\uu)$ and the equality holds true whenever  $E(\uu)$ is irrational. It is for instance the case of the Fibonacci sequence.  Nevertheless,  $E^*(\uu)$ and  $E(\uu)$ can coincide even if $E(\uu)$ is rational: it is the case of the Thue-Morse sequence.  
 
 While the value $E(\uu)$ takes into account repetitions of all factors of $\uu$,  $E^*(\uu)$ considers only repetitions of factors of length tending to infinity.  There is a huge literature devoted to questions situated between these two extremes. For example, Shur and Tunev~\cite{ShTu} construct a $d$-ary sequence $\uu$ whose all factors (except the trivial repetition of one letter in factors of the form $ a_1a_2\cdots a_{d-1}a_1$) have the exponent $< 1+\frac{1}{d-1} = RT(d)$. Other results of this flavour can be found in~\cite{BaCr, De, Sh}.

We provide a simple construction showing that  for any $d\in \mathbb{N}, d\geq 2,$ and any $\epsilon>0$,
there exists a~$d$-ary sequence $\uu$ with $E^*(\uu)<1+\epsilon$.
Therefore, if we denote $$RT^*(d) =\inf\{E^*(\uu): \uu \text{\ is a sequence over a $d$-ary alphabet}\}\,, $$
we have $RT^*(d) =1 \quad \text{for all $d\geq 2$}\,.$
Then we restrict our study to balanced sequences and  look for  the threshold
$$RTB^*(d) =\inf\{E^*(\vv): \vv \text{\ is a  balanced sequence over a $d$-ary alphabet}\}\,. $$
This threshold is bounded from below by $1+\frac{1}{2^{d-2}}$, see Corollary~\ref{coro:lower_bound}. 
It is known that $RTB^*(2)=RTB(2)=2 + \frac{\sqrt{5}+1}{2}\doteq3.618$ and it is reached by the Fibonacci sequence.
We introduce a new tool -- graphs of admissible tails -- for computation of the exact value of $RTB^*(d)$. Using it we obtain at once $RTB^*(d)$  for $ d \leq 10$, see Table~\ref{RTB*}. Let us point out that a similar result for $RTB(d)$ had been done step by step by several research teams.

Comparing the results with the minimal critical exponent of balanced sequences, we can see that 
$RTB^*(d)=RTB(d)$ for $d \in \{2,3,4,5\}$, but $RTB^*(d)<RTB(d)$ for larger $d$. Moreover, 
the precise values  $RTB^*(d)$ we have found for $d\leq 10$ suggest that  $RTB^*(d) < 1+ 
{q^d}$ for some positive $q <1$, whereas $RTB(d) > RT(d) = 1+ \frac{1}{d-1}$.   Looking into  Table~\ref{RTB*} we can see that

\medskip 

\centerline{ $RTB^*(6) < 1+\frac{1}{4}$, \  $RTB^*(7)< 1+\frac{1}{7}$, \  $RTB^*(8)< 1+\frac{1}{20}$, \  $RTB^*(9)< 1+\frac{1}{30}$,   \ 
$RTB^*(10)< 1+\frac{1}{68}$. }

\noindent The time and space complexity of our  algorithm allowed us to determine $RTB^*(d)$ only for $d\leq 10$. It seems that a new idea is  needed to extend Table~\ref{RTB*}  for  $d\geq 11$.  

%$$
%\begin{array}{rclcl}
%RTB^*(6) &<& 1+\frac{1}{4}&<&  1+\frac{1}{3}= RTB(6) \\
 %RTB^*(7) &<& 1+\frac{1}{7}&<&   1+\frac{1}{4}= RTB(7) \\
%RTB^*(8) &<& 1+\frac{1}{20}&<&   1+\frac{1}{5}= RTB(8) \\
 %RTB^*(9) &<&  1+\frac{1}{30}&<&   1+\frac{1}{6}= RTB(9) \\
% RTB^*(10) &<&  1+\frac{1}{68}&<&  1+\frac{1}{7}= RTB(10) \\
%\end{array}

\medskip

\section{Preliminaries}
\label{Section_Preliminaries}
An \textit{alphabet} $\A$ is a finite set of symbols called \textit{letters}.
A \textit{word} over $\A$ of \textit{length} $n$ is a string $u = u_0 u_1 \cdots u_{n-1}$, where $u_i \in \A$ for all $i \in \{0,1, \ldots, n-1\}$.
The length of $u$ is denoted by $|u|$. 
The set of all finite words over $\A$ together with the operation of concatenation forms a monoid, denoted $\A^*$.
Its neutral element is the \textit{empty word} $\varepsilon$ and we denote $\A^+ = \A^* \setminus \{\varepsilon\}$.
If $u = xyz$ for some $x,y,z \in \A^*$, then $x$ is a \textit{prefix} of $u$, $z$ is a \textit{suffix} of $u$ and $y$ is a \textit{factor} of $u$.
To any word $u$ over $\A$ with cardinality $\#\A = d$,  we assign its \textit{Parikh vector} $\vec{V}(u) \in \N^{d}$ defined as $(\vec{V}(u))_a = |u|_a$ for all $a \in \A$, where $|u|_a$ is the number of letters $a$ occurring in $u$.

A \textit{sequence} over $\A$ is an infinite string $\uu = u_0 u_1 u_2 \cdots$, where $u_i \in \A$ for all $i \in \N$. The notation $\A^{\mathbb{N}}$ stands for the set of all sequences over $\A$. 
We always denote sequences by bold letters. 
The shift operator $\sigma$ maps any sequence $\uu = u_0 u_1 u_2 \cdots$ to the sequence $\sigma(\uu) = u_1 u_2 u_3 \cdots$. 

A sequence $\uu$ is \textit{eventually periodic} if $\uu = vwww \cdots = vw^\omega$ for some $v \in \A^*$ and $w \in \A^+$.
It is \textit{periodic} if $\uu=w^{\omega}$. In both cases, the number $|w|$ is a \textit{period} of $\uu$. We write $\Per\, \uu$ for the minimal period of $\uu$.
If $\uu$ is not eventually periodic, then it is \textit{aperiodic}.
A \textit{factor} of $\uu = u_0 u_1 u_2 \cdots$ is a word $y$ such that $y = u_i u_{i+1} u_{i+2} \cdots u_{j-1}$ for some $i, j \in \N$, $i \leq j$. 
The number $i$ is called an \textit{occurrence} of the factor $y$ in $\uu$.
In particular, if $i = j$, the factor $y$ is the empty word $\varepsilon$ and any index $i$ is its occurrence.
If $i=0$, the factor $y$ is a \textit{prefix} of $\uu$.
If each factor of $\uu$ has infinitely many occurrences in $\uu$, the sequence $\uu$ is \textit{recurrent}.
Moreover, if for each factor the distances between its consecutive occurrences are bounded, $\uu$ is \textit{uniformly recurrent}. 

The \textit{language} $\mathcal{L}(\uu)$ of a sequence $\uu$ is the set of all its factors. 
A factor $w$ of $\uu$ is \textit{right special} if $wa, wb$ are in $\mathcal{L}(\uu)$ for at least two distinct letters $a,b \in \A$.
A \textit{left special} factor is defined symmetrically.
A~factor is \textit{bispecial} if it is both left and right special.

A~sequence $\uu\in \A^{\mathbb{N}}$ is \textit{balanced} if for every letter $a \in \A$ and every pair of factors $u,v \in {\mathcal L}(\uu)$ with $|u|=|v|$, we have $|u|_a-|v|_a\leq 1$.
Every recurrent balanced sequence over any alphabet  is uniformly recurrent, see \cite{DolceDP21}. 
An aperiodic balanced sequence  over a~binary alphabet  is   called \textit{Sturmian},
 \cite{MoHe40}. There exist many equivalent definitions of Sturmian sequences, see for instance \cite{BeSe02}.

A \textit{morphism} over $\A$ is a mapping $\psi: \A^* \to \A^*$ such that $\psi(uv) = \psi(u)\psi(v)$  for all $u, v \in \A^*$.
Morphisms can be naturally extended to $\A^{\mathbb{N}}$ by setting
$\psi(u_0 u_1 u_2 \cdots) = \psi(u_0) \psi(u_1) \psi(u_2) \cdots\,$.
A \textit{fixed point} of a morphism $\psi$ is a sequence $\uu$ such that $\psi(\uu) = \uu$.
\begin{example}
\label{ex:FiboDef}
The most famous Sturmian sequence is  the \textit{Fibonacci sequence}
$$
\uu_f = {\tt babbababbabbababbababb}\cdots\,,
$$
defined as the fixed point of the morphism $f: {\tt b} \mapsto {\tt ba}$, ${\tt a} \mapsto {\tt b}$. The critical exponent of $\uu_f$ is $2+\frac{1+\sqrt{5}}{2}$. As shown by Carpi and de Luca~\cite{CaLu2000}, it is the least critical exponent for Sturmian sequences, i.e., $RTB(2) = E(\uu_f)$. 
\end{example}

Consider a factor $w$ of a recurrent sequence $\uu = u_0 u_1 u_2 \cdots$.
Let $i < j$ be two consecutive occurrences of $w$ in $\uu$.
Then the word $u_i u_{i+1} \cdots u_{j-1}$ is a \textit{return word} to $w$ in $\uu$.
The set of all return words to $w$ in $\uu$ is denoted by $\R_\uu(w)$.
If $\uu$ is uniformly recurrent, the set $\R_\uu(w)$ is finite for each factor $w$.
If $w$ is a prefix of $\uu$, then $\uu$ can be written as a concatenation $\uu=r_{d_0}r_{d_1}r_{d_2} \cdots$ of return words to $w$. 
The sequence
$\dd_\uu(w) = d_0d_1d_2 \cdots$
 over the alphabet of cardinality $\# \R_\uu(w)$ is called the  \textit{derived sequence} of $\uu$ to $w$.
 If $w$ is not a prefix, then we construct the derived sequence in an analogous way starting from the first occurrence of $w$ in $\uu$. 
 The concept of derived sequences was introduced by Durand~\cite{Dur98}.

\section{Asymptotic critical exponent and its relation to return words}
\label{Section_CriticalExponent}

In \cite{DDP21} a handy formula for computation of the critical exponent and asymptotic critical exponent of uniformly recurrent sequences is deduced.  It uses the notion of return words to a factor of a sequence.
\begin{theorem}[\cite{DDP21}]
\label{Prop_FormulaForCR}
Let $\uu$ be a uniformly recurrent aperiodic sequence\footnote{Arseny M. Shur in private communication pointed out to us that the same formula remains valid for recurrent aperiodic sequences.}.
Let $(u_n)$ be a sequence of all bispecial factors ordered by their length.
For every $n \in \N$, let $r_n$ be a shortest return word to $u_n$ in $\uu$.
Then
$$
\CR(\uu) = 1 + \sup\limits_{n \in \N} \left\{ \frac{|u_n|}{|r_n|} \right\}
\qquad \text{and} \qquad
\CR^*(\uu) = 1 + \limsup\limits_{n \to \infty}  \frac{|u_n|}{|r_n|} .
$$
\end{theorem}

\begin{theorem} Let $\mathcal{A}$ be a finite alphabet of size at least 2.  Then    $$\inf\{\, E^*(\uu)\ : \ \uu \in \mathcal{A}^\mathbb{N}, \ \uu \text{\  uniformly recurrent}\,\} =  1. $$
Consequently, $RT^*(d) = 1$ for every $d\geq 2$.  
\end{theorem}
\begin{proof}  It is enough to prove the statement for the alphabet $\mathcal{A} = \{0,1\}$. For  every Fibonacci number\footnote{The Fibonacci sequence is defined recursively: $F_0=0, F_1 = 1$, and $F_{n+2} = F_{n+1}+F_n$ for every $n \in  \N$.  } $F_k$, with $k\geq 7$,  we construct a binary sequence $\uu=\uu^{(k)}$ such that every bispecial factor $u$ of length at least $3(k+1)$  and every  return word $r$ to $u$ in $\uu$ satisfy $\frac{|u|}{|r|} < \frac{2}{F_k-3}$. For the sequence $\uu^{(k)}$, the second formula of Theorem  \ref{Prop_FormulaForCR} gives  $E^*\bigl({\uu^{(k)}}\bigr) \leq 1+ \frac{2}{F_k-3}$, which implies  the theorem.  Construction of $\uu^{(k)}$ follows.

Let $\mathcal{D} = \{0,1,\ldots, d-1\}$, where $d =2 \lfloor \frac12 F_k\rfloor \in \{F_k-1, F_k\}$.  By \cite{DOPS2022} there exists a balanced (and hence uniformly recurrent)  $d$-ary  sequence   $\vv$ having $E(\vv) = \frac{d-1}{d-2}$. 
Zeckendorf's theorem~\cite{Ze72} says that  every $i \in \mathcal{D}$ can be written in the form $i=\sum_{n=2}^{k-1} c_n F_n$, where
$c_{k-1}c_{k-2}\cdots c_2$ is a word over the alphabet $\{0,1\}$, which does not contain two consecutive $1$'s. We denote the $(k-2)$-tuple $c_{k-1}c_{k-2}\cdots c_2$ representing $i$ by  $(i)_{Fib}$ and define a  morphism  $\psi: \mathcal{D}^* \mapsto \{0,1\}^*$ and the binary sequence  $\uu$  as follows: 
$$\psi(i) = 110(i)_{Fib} \  \text{ for every } i \in \mathcal{D}\quad \text{and} \quad \uu = \uu^{(k)}=\psi(\vv).$$ The morphism $\psi$ is uniform since the image of any letter $i$ by $\psi$ has length $|\psi(i)| = k+1$.  Moreover,  $\uu $ is uniformly recurrent as $\vv$ is uniformly recurrent.  $\psi$ is  a coding since the  factor $110$ occurs only as a prefix of each $\psi(i)$.  Hence any factor $u\in \mathcal{L}(\uu)$ longer than $k+2$ can be written uniquely in  the form $u= u^{(L)}\psi(v) u^{(R)}$, where $v \in \mathcal{L}(\vv)$,   $u^{(L)}$ is a proper suffix of $\psi(i)$ and $u^{(R)}$  is a proper prefix  of $\psi(j)$  for some $i, j \in \mathcal{D}$.  
Obviously, if $u$ is a left special factor of $\uu$, i.e., $0u$ and $1u$ belong to the language of $\uu$, then $v$ is a left special factor of $\vv$. An analogous statement  is true  for  right special factors of $\uu$.

Let $u$ be a bispecial factor   of $\uu$  with  $|u| \geq 3(k+1)$ and $r$ be a return word to $u$ in $\uu$. Then there exist  a bispecial factor $v\in \mathcal{L}(\vv)$ of length  $|v| \geq 2$ and $ s \in \mathcal{L}(\vv)$ such that $u = u^{(L)}\psi(v) u^{(R)}$ and $ru = u^{(L)}\psi(sv) u^{(R)}$.  Obviously,   $s$ is a return word or concatenation of several return words to $v$ in $\vv$  and $|r| = |\psi(s)| = (k+1)|s|$.  As $E(\vv) = 1+\frac{1}{d-2}$,  Theorem \ref{Prop_FormulaForCR} implies $\frac{|v|}{|s|} \leq \frac{1}{d-2}< \frac{1}{F_k-3} $.   Hence
$$
\frac{|u|}{|r|} = \frac{|u^{(L)}| + (k+1)|v| + |u^{(R)}|}{ (k+1)|s|} \leq \frac{2k + (k+1)|v|}{(k+1)|s|} <  \frac{2|v|}{|s|} <   \frac{2}{F_k-3} \,,
$$ 
as we wanted to show. 
\end{proof}

\section{Sturmian sequences}
\label{sec:sturmian}

Sturmian sequences, i.e.,  aperiodic balanced  sequences over a binary alphabet,  are a principal tool in the study of balanced sequences over arbitrary alphabets. 
In the sequel, we will restrict our consideration to standard sequences. Let us recall that a Sturmian sequence $\uu$ is called a {\em standard sequence}
if both sequences ${\tt a}\uu$ and ${\tt b}\uu$ are Sturmian. For each Sturmian sequence there exists a unique standard sequence having the same language and thus the same critical exponent and  the asymptotic critical exponent. Sturmian sequences have well defined frequencies of letters. Let us recall that 
the \emph{frequency} of a letter $a$ in a sequence $\uu$ is the limit $\rho_a(\uu)=\lim_{n\to\infty} \frac{|u_0\cdots u_{n-1}|_a}{n}$ if it exists.

We use the characterization of standard sequences by their directive sequences. To introduce them, we recall two morphisms
$$
G = \
\left\{ \,
\begin{aligned}
{\tt a} & \to {\tt a} \\
{\tt b} & \to {\tt ab} \,
\end{aligned}
\right.
\quad \text{and} \quad
D = \
\left\{ \,
\begin{aligned}
{\tt a} & \to {\tt ba} \\
{\tt b} & \to {\tt b} \,
\end{aligned}
\right. .
$$

\begin{proposition}[{\cite{JuPi02}}]
\label{Prop_Standard}
For every standard sequence $\uu$ there is a uniquely given \emph{directive sequence} $\mathbf{\Delta} = \Delta_0 \Delta_1 \Delta_2 \cdots \in \{ G, D \}^\N$ of morphisms and a sequence $(\uu^{(n)})$ of standard sequences such that
$$
\uu = {\Delta_0 \Delta_1 \cdots \Delta_{n-1}} \left( \uu^{(n)} \right) \, \ \text{for every } \ n \in \N\,.
$$
Both $G$ and $D$ occur in the sequence $\mathbf{\Delta}$ infinitely often.
\end{proposition}

If $\Delta_0 = D$, then  $\mathtt{b}$ is the most frequent letter in $\uu$. Otherwise,  $\mathtt{a}$ is the most frequent letter in $\uu$.  We adopt the convention that $\rho_{\tt b}(\uu) > \rho_{\tt a}(\uu)$ and thus the directive sequence of $\uu$ starts with $D$. Let us write this sequence in the run-length encoded form  $\mathbf{\Delta} = D^{a_1} G^{a_2} D^{a_3} G^{a_4} \cdots$,
where all integers $a_n$ are positive. Then the number $\theta$ having the continued fraction expansion 
$\theta = [0, a_1, a_2, a_3, \ldots]$ equals the ratio 
$\frac{\rho_{\tt a}(\uu)}{\rho_{\tt b}(\uu)}$ (see \cite{BeSe02}) and $\theta$ is called the {\em slope} of $\uu$.

The convergents to the continued fraction of $\theta$,  usually  denoted $\frac{p_N}{q_N}$, have a close relation to return words in a Sturmian sequence. Recall that the sequences $(p_N)$ and $(q_N)$ both satisfy the recurrence relation 
\begin{equation}\label{eq:convergents}
X_{N+1}=a_{N+1}X_N+X_{N-1}
\end{equation}
with initial conditions $p_{-1}=1$, $p_0=0$ and $q_{-1}=0$, $q_0=1$.  Two consecutive convergents satisfy $p_{N}q_{N-1}-p_{N-1}q_N = (-1)^{N+1}$ for every $N\in \N$.

Vuillon~\cite{Vui01} showed that an infinite recurrent sequence $\uu$ is Sturmian if and only if each of its factors has exactly two return words. Moreover, the derived sequence of a Sturmian sequence to any of its factors is also Sturmian. 

All bispecial factors of any standard sequence $\uu$ are its prefixes. So, one of the return words to a bispecial factor of $\uu$ is a prefix of $\uu$.

\begin{proposition}[{\cite{DvMePe20}}]\label{prop:returnWords}
\label{ParikhRSB1}
Suppose that $\uu$ is a standard sequence with slope $\theta = [0, a_1,a_2, a_3, \ldots]$ and $z$ is a bispecial factor of $\uu$. Let $r$ (resp., $s$) denote the return word to $z$ which is (resp., is not) a prefix of $\uu$. Then
\begin{enumerate}
\item there exists a unique pair $(N,m) \in \N^2$ with $0 \le m < a_{N+1}$ such that the Parikh vectors of $r$, $s$, and $z$ are respectively
$$
\vec{V}(r) = \begin{pmatrix} p_N \\ q_N \end{pmatrix},
\; \; 
\vec{V}(s) = \begin{pmatrix} m \, p_{N} + p_{N-1} \\ m \, q_{N} + q_{N-1} \end{pmatrix},
\; \; 
\vec{V}(z) = \vec{V}(r) + \vec{V}(s) - \begin{pmatrix} 1 \\ 1 \end{pmatrix};
$$
\item the slope of the derived sequence $\dd_\uu(z)$ is
$${\theta}'=[0, a_{N+1} - m, a_{N+2}, a_{N+3}, \ldots].$$
\end{enumerate}
\end{proposition}

To express the length of  bispecial factors and their return words in a Sturmian sequence with slope $\theta$, we use  for each $N\in \N$ the notation 

\medskip

\centerline{$Q_N: =p_N+q_N$, where $\frac{p_N}{q_N}$ is the $N^{th}$ convergent to $\theta$,}
i.e., the sequence $(Q_N)$ satisfies~\eqref{eq:convergents} with initial conditions $Q_{-1}=1=Q_0$.

\begin{remark}\label{Lenz} The formulae for computation of $E(\uu)$  and  $E^*(\uu)$ for a Sturmian sequence $\uu$  with slope $\theta = [0,a_1,a_2,a_3, \ldots]$ were provided by Daminik and Lenz in \cite{DaLe2002}, see also~\cite{CaLu2000}.
 $$ E(\uu) = 2+ \sup_{N\in \N} \Bigl\{a_{N+1} + \tfrac{Q_{N-1}-2}{Q_{N-2}} \Bigr\}\quad \text{and}\quad  E^*(\uu) = 2+ \limsup_{N\to \infty} \Bigl\{a_{N+1} + \tfrac{Q_{N-1}}{Q_{N-2}}  \Bigr\}\,. 
 $$
These formulae  can be deduced easily  
using Proposition \ref{prop:returnWords}  and Theorem \ref{Prop_FormulaForCR}.  

%If a   bispecial factor $z \in \mathcal{L}(\uu)$ is associated  to the pair $(N,0)$, then  the shortest return $v$ word   to $z$ is $s$ and  we have  to consdier the ratio $$\frac{|z|}{|v|} = \frac{Q_N+Q_{N-1}-2}{Q_{N-1}} = \frac{a_NQ_{N-1} + Q_{N-2} + Q_{N-1}-2}{Q_{N-1}}
 %=  1+a_{N} + \frac{Q_{N-2}-2}{Q_{N-1}}.$$ 
 
 %If a bispecial factor $z$ is associated with a pair $(N, m)$ with  $0< m < a_{N+1}$, then the shortest return word to $z$ is $r$ and the ratio we have to consider is    $$\frac{|z|}{|v|} = \frac{(m+1)Q_N+Q_{N-1}-2}{Q_{N}} = m+1 + \frac{Q_{N-1}-2}{Q_{N}} \leq a_{N+1}
 %+\frac{Q_{N-1}-2}{Q_{N}}.$$
 %Theorem \ref{Prop_FormulaForCR} immediately implies 
\end{remark}

In the sequel, it will be necessary to recognize which vectors are Parikh vectors of some factors in a~Sturmian sequence. The answer follows.
\begin{lemma}[{\cite{DDP21}}]
\label{lem_kl}
Let $\uu$ be a Sturmian sequence with slope $\theta=\frac{\rho_{\tt a}(\uu)}{\rho_{\tt b}(\uu)}<1$.
Denote $\delta = \theta^{-1}$. Then $\uu$ contains a factor $w$ such that $|w|_{\tt b} = k$ and $|w|_{\tt a} = \ell$ if and only if
\begin{equation}
\label{pocet}
 |\ell\delta -k| < \delta +1.
 \end{equation}
\end{lemma}

\section{Balanced sequences }
\label{Section_BalancedSturmian1}

In 2000 Hubert~\cite{Hubert00} characterized balanced sequences over alphabets of cardinality bigger than 2 in terms of Sturmian sequences, \textit{colourings}, and \textit{constant gap sequences}.

\begin{definition}\label{def:colouring}
Let $\uu$ be a sequence over $\{\mathtt{a},\mathtt{b}\}$, $\yy$ and $\yy'$ be arbitrary sequences. The \emph{colouring} of $\uu$ by $\yy$ and $\yy'$ is the sequence $\vv = \barva( \uu, \yy, \yy')$ obtained from $\uu$ by replacing the subsequence of all $\mathtt{a}$'s with $\yy$ and the subsequence of all $\mathtt{b}$'s with $\yy'$.
%Let $\uu$ be a Sturmian sequence over the alphabet $\{ {\tt a}, {\tt b}\}$, and ${\yy}, {\yy'}$ be two constant gap sequences over two disjoint alphabets $\A$ and $\B$.
%Let a~sequence $\vv$ be obtained from $\uu$ by replacing the ${\tt a}$'s in $\uu$ step by step by letters of $\yy$ and replacing the ${\tt b}$'s in $\uu$ step by step by letters of $\yy'$. The sequence $\vv$ is called \emph{colouring} of $\uu$ by $\yy$ and $\yy'$ and denoted $\vv = \barva( \uu, \yy, \yy')$.
\end{definition}

\begin{definition}
\label{Def_ConstantGap}
A~sequence $\yy$ is a \emph{constant gap sequence} if for each letter $a$ occurring in $\yy$ the distance between any consecutive occurrences of $a$ in $\yy$ is  constant.
\end{definition}

There is a rich literature on a notion equivalent to constant gap sequences, the so-called exact covering systems, see~\cite{Znam1969, PoSch2002, GoGrBrSh2019}.

\begin{example}
\label{ex:sequences} The periodic  sequences  $\yy =(\texttt{34})^\omega$ and  $\yy'=(\texttt{0102})^\omega$ are constant gap sequences over binary and ternary alphabet, respectively. 
The sequence $\vv = \barva( \uu_f, \yy, \yy')$, where $\uu_f$ is the Fibonacci sequence defined in Example~\ref{ex:FiboDef}, looks as follows:
\begin{align*}
\uu_f &= \mathtt{babbababbabbababbababbabbababba}\cdots \\
\vv &= \mathtt{ \textcolor{red}{0}\textcolor{blue}{3}\textcolor{red}{10}\textcolor{blue}{4}\textcolor{red}{2}\textcolor{blue}{3}\textcolor{red}{01}\textcolor{blue}{4}\textcolor{red}{02}\textcolor{blue}{3}\textcolor{red}{0}\textcolor{blue}{4}\textcolor{red}{10}\textcolor{blue}{3}\textcolor{red}{2}\textcolor{blue}{4}\textcolor{red}{01}\textcolor{blue}{3}\textcolor{red}{02}\textcolor{blue}{4}\textcolor{red}{0}\textcolor{blue}{3}\textcolor{red}{10}\textcolor{blue}{4}} \cdots
\end{align*}
\end{example}

\begin{theorem}[{\cite{Hubert00}}]
\label{Hubert}
A recurrent aperiodic sequence $\vv$ is balanced if and only if $\vv = \barva( \uu, \yy, \yy')$ for some Sturmian sequence $\uu$ and constant gap sequences ${\yy}, {\yy'}$ over two disjoint alphabets.
\end{theorem}

\medskip
Let $\A,\B$ be two disjoint alphabets. The ``discolouration map'' $\pi$ is defined for any word or sequence over $\A\cup \B$; it replaces all letters from $\A$ by $\mathtt{a}$ and all letters from $\B$ by $\mathtt{b}$. If  $\vv = \barva( \uu, \yy, \yy')$, where $\yy\in \A^{\mathbb{N}}$, $\yy'\in \B^{\mathbb{N}}$, then $\pi(\vv)=\uu$ and $\pi(v) \in \mathcal{L}(\uu)$ for every $v \in \LL(\vv)$.

\begin{corollary}[{\cite{DDP21}}]
\label{cor:cyclicshift}
Let $\vv = \barva( \uu, \yy, \yy')$ and $u \in \mathcal{L}(\uu)$.
For any $i,j \in \mathbb N$, the word $v$ obtained by colouring $u$ with shifted constant gap sequences 
$\sigma^i(\yy)$ and $\sigma^j(\yy')$ is in $\mathcal{L}(\vv)$. In particular, if a  Sturmian sequence $\widetilde{\uu}$ has the same language as  $\uu$, then, $E(\vv) = E(\widetilde{\vv})$ and  $E^*(\vv) = E^*(\widetilde{\vv})$.
\end{corollary}
\begin{example}
Let $\uu_f, \vv, \yy$ and $\yy'$ be as in Example~\ref{ex:sequences}.
Let $u = {\tt bab}$.
The reader is invited to check that all words ${\tt 031}, {\tt 130}, {\tt 032}, {\tt 230}, {\tt 041}, {\tt 140}, {\tt 042}, {\tt 240}$ are factors of $\vv$.
\end{example}

As a consequence of the previous corollary when studying the asymptotic critical exponent of balanced sequences, we can limit our consideration to colourings of standard Sturmian sequences.

\begin{proposition}[{\cite{DDP21}}]\label{prop:bounds} Let $\vv=\colour(\uu,\yy,\yy')$ and $\theta=[0,a_1,a_2,a_3,\ldots]$.
One has:
\begin{enumerate}
    \item $E(\vv) \geq E^*(\vv) \geq 1 +\frac{1}{\Per\, {\yy}\cdot \Per\, {\yy'}}\,.$
    \item $E^*(\vv)$ depends on $\Per\, {\yy}$ and $\Per\, {\yy'}$, not on the structure of $\yy$ and $\yy'$.
    \item $E^*(\vv)$ is finite $\Leftrightarrow$ $E^*(\uu)$ is finite $\Leftrightarrow$ $(a_n)$ is a bounded sequence.
\end{enumerate}
\end{proposition}

Hoffman et al. proved that any $d$-ary constant gap sequence satisfies $\Per\, \yy \leq 2^{d-1}$ (see Theorem 2 in~\cite{Hof2011} based on Corollary 2 from~\cite{Si85}). Therefore, the first item of Proposition~\ref{prop:bounds} leads to the following corollary.
\begin{corollary}\label{coro:lower_bound}
$RTB^*(d)\geq 1+ \frac{1}{2^{d-2}}$ for every positive integer $d$.
\end{corollary}

\medskip The {\bf main task} we solve in the paper is the following: for given $P$ and $P'$  find  a  balanced sequence $\vv$ such that its asymptotic critical exponent  $E^*(\vv)$ has  the least value among all balanced sequences which arise as  colouring of Sturmian sequences  by two constant gap sequences $\yy$ and $\yy'$ with  $P = \Per \,\yy$ and $P' = \Per\,\yy'$.

Having a method for solving this task, we are able to determine $RTB^*(d)$ for a fixed $d$. We apply the method to all pairs of  $P, P'$ such that $P= \Per\,\yy $ for some constant gap sequence $\yy$ over $d_{\tt a}$-ary alphabet and  $P'= \Per\,\yy' $ for some constant gap  sequence $\yy'$ over $d_{\tt b}$-ary alphabet, where $d_{\tt a}+ d_{\tt b} = d$. 

For a fixed $d$ there are only finitely many pairs $P,P'$ with the described property. It seems that to find the periods of  all constant gap sequences over a $d$-ary alphabet is a  difficult problem.
Nevertheless, it is known (see \cite{Schnabel2015})  that constant gap sequences over $d$ letters with $d \leq 12$ may be obtained by {\em{interlacing}}. 
\begin{definition}
Let $\mathcal{A}^{(0)}, \mathcal{A}^{(1)}, \ldots, \mathcal{A}^{(k-1)}$  be mutually disjoint alphabets and let   $\yy^{(i)} = y_0^{(i)}y_1^{(i)}y_2^{(i)}\cdots$ be  a constant gap sequence over the alphabet   $\mathcal{A}^{(i)}$ for every $i \in\{0,1,\ldots, k-1\}$.
The {\em interlacing} of  $\yy^{(0)}, \yy^{(1)}, \ldots, \yy^{(k-1)}$ is the sequence  $\yy =y_0y_1y_2\cdots$, where $y_{kn+j} = y^{(j)}_n $  for every $n\in \N$  and $j \in \{0,1,\ldots, k-1\}$. 
 
 In other words, the interlacing of $\yy^{(0)}, \yy^{(1)}, \ldots, \yy^{(k-1)}$ is a sequence obtained by listing  step by step the first letters of $\yy^{(0)}, \yy^{(1)}, \ldots, \yy^{(k-1)}$, then the second letters, the third letters etc. 
\end{definition}

The interlacing $\yy$ of  $\yy^{(0)}, \yy^{(1)}, \ldots, \yy^{(k-1)}$ satisfies $$\Per\, {\yy}=k\cdot \lcm \{\Per{\,\yy^{(0)}}, \Per{\,\yy^{(1)}}, \ldots, \Per{\,\yy^{(k-1)}}\}.$$

\begin{example}
The interlacing of $\yy^{(0)}= (\tt 0102)^\omega$ and $\yy^{(1)}=(\tt 34)^\omega$ equals $(\tt 03140324)^{\omega}$. The reader may easily verify that it is again a constant gap sequence and its period equals $2\cdot \lcm \{\Per{\,\yy^{(0)}}, \Per{\,\yy^{(1)}}\}=8$.
\end{example}

\begin{remark}\label{rem:periods_const_gap}
Using the fact that all constant gap sequences with at most 12 letters are obtained by interlacing, it is not difficult to verify that the periods of constant gap sequences over alphabet of size $d$ are as follows:
$$\begin{array}{l|l}
      d& \text{period}  \\
      \hline
    1 &1 \\
    2&2\\
    3 & 3,4\\
4& 4, 6, 8\\
5& 5, 6, 8, 9, 12, 16\\
 6& 6, 8, 10, 12, 16, 18, 24, 32\\
7& 7, 8, 9, 10, 12, 15, 16, 18, 20, 24, 27, 32, 36, 48, 64\\
8& 8, 10, 12, 14, 16, 18, 20, 24, 30, 32, 36, 40, 48, 54, 64, 72, 96, 128\\
9 & 9, 10, 12, 14, 15, 16, 18, 20, 21, 24, 25, 27, 28, 30, 32, 36, 40, 45, \\
& 48, 54, 60, 64, 72, 80, 81, 96, 108, 128, 144, 192, 256
\end{array}
$$
\end{remark}

\section{Formula for the asymptotic critical exponent of balanced sequences} \label{sec:formulaE*}
In this section,   $\theta = [0,a_1,a_2,a_3, \ldots]$ denotes the continued fraction of the slope of a standard Sturmian sequence $\uu$,  $P = \Per \,\yy$ and $P' = \Per \,\yy'$  are  periods of two constant gap sequences  and $\vv = \colour(\uu, \yy, \yy')$. 
The computation of the asymptotic critical exponent of $\vv$ is based on the knowledge of bispecial factors of the Sturmian sequence $\uu$ and their return words. To provide an explicit formula let us first fix some notation. 

\medskip 
\noindent {\bf Convention:}  Given positive $c,d \in \N $. 

\begin{itemize}
    \item   If $a_1 = b_1 \bmod c  \text{ and } a_2 = b_2 \bmod d$,
     we write 
$$\left( \begin{smallmatrix} a_1 \\ a_2\end{smallmatrix} \right)  =  \left( \begin{smallmatrix} b_1 \\ b_2  \end{smallmatrix} \right) \bmod \left( \begin{smallmatrix} c \\ d \end{smallmatrix} \right).$$
    \item If $\left( \begin{smallmatrix} a_{1i} \\ a_{2i}\end{smallmatrix} \right)  =  \left( \begin{smallmatrix} b_{1i} \\ b_{2i}  \end{smallmatrix} \right) \bmod \left( \begin{smallmatrix} c \\ d \end{smallmatrix} \right)$  for $i=1, 2$, we write 
$$\left( \begin{smallmatrix} a_{11} & a_{12} \\ a_{21} & a_{22} \end{smallmatrix} \right)  =  \left( \begin{smallmatrix} b_{11} & b_{12} \\ b_{21} & b_{22} \end{smallmatrix} \right) \bmod \left( \begin{smallmatrix} c \\ d \end{smallmatrix} \right).$$
    
\end{itemize}

  A formula for computation of $E^*(\vv)$ is deduced in {\cite{DDP21}}. To keep this paper self-contained, we provide a sketch of its proof. It is based on the following simple observation. 

\begin{observation}\label{obs:special}
Let  $v$ be a factor of  $\vv$ and $u=\pi(v)$, where $|u|_{\tt a} \geq P, \ |u|_{\tt b}\geq P'$. 
Then 

\begin{enumerate}
\item $v$ is bispecial in $\vv$ if and only if $u$ is bispecial in  $\uu$;

\item if $i\in \N$ is an occurrence of $v$ in $\vv$, then $i$ is an occurrence of $u$ in $\uu$.
\end{enumerate}
\end{observation}

 Denote $ A_N =\left( \begin{smallmatrix} p_{N-1}  & p_N\\  q_{N-1} & q_{N} \end{smallmatrix} \right), \ \ \delta_N = [a_{N+1}, a_{N+2}, \ldots ]$.

 \medskip 
Let $(N,m)$ be a pair associated in  Proposition~\ref{prop:returnWords} to a  bispecial factor of $\uu$. We assign to $(N,m)$ the sets:

$\mathcal{S}_1(N,m)=\bigl\{ \left( \begin{smallmatrix}  \ell \\ k \end{smallmatrix} \right) : A_N\left( \begin{smallmatrix} 1 & 0 \\  m & 1 \end{smallmatrix} \right)  \left( \begin{smallmatrix}  \ell \\ k \end{smallmatrix} \right) = \left( \begin{smallmatrix}  0\\ 0 \end{smallmatrix} \right) \mod \left( \begin{smallmatrix}  P \\ P' \end{smallmatrix} \right) \bigr\}\,;$

$\mathcal{S}_2(N,m)=\{ \left( \begin{smallmatrix}  \ell \\ k \end{smallmatrix} \right) : |\ell(\delta_N-m) -k| < \delta_N -m+1 \text{ and } k+\ell >0  \}\,;$  

$\mathcal{S}(N,m) = \mathcal{S}_1(N,m) \cap \mathcal{S}_2(N,m)\,.   $

\medskip

\begin{proposition}\label{vzorec} Let 
$\uu$ be a Sturmian sequence  with slope $\theta = [0,a_1,a_2,a_3,\ldots]$ and $\yy$ and $\yy'$ be two constant gap sequences. Put 
\begin{equation}\label{Eq:Phi}
\Phi_N : = \max\left\{\frac{1+m +\frac{Q_{N-1}}{Q_N}}{k + \ell m +\ell \frac{Q_{N-1}}{Q_N}}\ : \    \left( \begin{smallmatrix}  \ell \\ k \end{smallmatrix} \right)  \in \mathcal{S}(N,m) \  \text{and} \  0\leq m <  a_{N+1} \right\}\,. 
\end{equation}
Then   the asymptotic critical exponent of the  balanced sequence $\vv = \colour(\uu, \yy,\yy')$ equals $E^*(\vv) =1+ \limsup\limits_{N\to \infty} \Phi_N$.  

\end{proposition}

\noindent

\begin{proof} 
Let $v$ be a bispecial factor of $\vv$ satisfying assumptions of Observation~\ref{obs:special} and $i<j$ be two consecutive occurrences of $v$ in $\vv$  and $w =v_iv_{i+1}\cdots v_{j-1}$.  By Item 1 of Observation \ref{obs:special} the factor $u= \pi(v)$ is bispecial in  $\uu$. Let $(N,m)$ be the pair associated by Proposition~\ref{prop:returnWords} with  $u$  and $r$ and $s$ be two return words to $u$ in $\uu$. By Item 2 of Observation \ref{obs:special} the projection $\pi(w)$  is concatenated from $\ell$ return words $s$  and $k$ return words $r$ for some $ \ell, k \in \N, k+\ell >0.$  Obviously, the vector $\left( \begin{smallmatrix}  \ell \\ k \end{smallmatrix} \right)$ is the Parikh vector of the derived sequence $d_{\uu}(u)$. By Proposition~\ref{prop:returnWords} the slope $\theta'$ of the derived sequence is  $[0, a_{N+1}-m, a_{N+2}, a_{N+3}, \ldots]$. The inverse of $\theta'$  is $\delta_N - m= [ a_{N+1}-m, a_{N+2}, a_{N+3}, \ldots]$. By Lemma \ref{lem_kl}  the pair $\ell, k$ satisfies the inequality    $|\ell(\delta_N-m) -k| < \delta_N -m+1 $.
Hence $\left( \begin{smallmatrix}  \ell \\ k \end{smallmatrix} \right)$ belongs to $\mathcal{S}_2(N,m)$.

Since $i$ and $j$ are occurrences of $v$ in $\vv$, the number of letters ${\tt a}$, resp. ${\tt b}$ in $\pi(w)$ is a multiple of $P$, resp. $P'$. Using the corresponding  Parikh vectors we have $\vec{V}(\pi(w)) =\ell   
 \vec{V}(s) + k\vec{V}(r) = \left( \begin{smallmatrix} 0 \\ 0  \end{smallmatrix} \right) \bmod \left( \begin{smallmatrix} P \\ P' \end{smallmatrix} \right)$. By Proposition~\ref{prop:returnWords}, 
 $\ell   
 \vec{V}(s) + k\vec{V}(r) =   A_N\left( \begin{smallmatrix} 1 & 0 \\  m & 1 \end{smallmatrix} \right)  \left( \begin{smallmatrix}  \ell \\ k \end{smallmatrix} \right)$, hence $\left( \begin{smallmatrix}  \ell \\ k \end{smallmatrix} \right)$ belongs to  $\mathcal{S}_1(N,m)$ and thus to  $\mathcal{S}(N,m)$. 
 
Let us evaluate the  ratio  $\frac{|v|}{|w|}$. We abbreviate  $ x_N = \frac{Q_{N-1}}{Q_N}$.  By Proposition~\ref{prop:returnWords} 
$$
\frac{|v|}{|w|} = \frac{|r| +  |s| -2}{k|r|+ \ell |s|}=
\frac{(1+m)Q_N + Q_{N-1} -2}{(k+\ell m)Q_N + \ell Q_{N-1}} 
=\frac{1+m + x_N}{k +\ell m + \ell x_{N}} - \frac{2}{|w|}. 
$$
The previous equality is valid for each factor  $w$ occurring between two occurrences of $v$ in $\vv$. If $w$ is a~shortest return word to $v$, then 
$$ \frac{|v|}{|w|} = \max\left\{ \frac{1+m + x_N}{k +\ell m + \ell x_{N}}: \left( \begin{smallmatrix}  \ell \\ k \end{smallmatrix} \right)  \in \mathcal{S}(N,m) \right\}-\frac{2}{|w|}\,.$$
Hence  $\Phi_N$  expresses  (up to the subtracted fraction $\frac{2}{|w|}$) the maximal value of the ratio $|v|/|w|$  among all possible  pairs $(N,m)$ with a fixed $N$. Since the length $|w|$ tends to infinity with growing $N$, Theorem \ref{Prop_FormulaForCR} concludes the proof.
\end{proof}

If the slope of a Sturmian sequence is quadratic irrational,  then the asymptotic critical exponent can be computed explicitly. 
The algorithm is explained in details in~\cite{DDP21}. We implemented it and throughout this paper we use our computer program which for a given eventually periodic continued fraction $\theta$ and a pair $P=\Per\, \yy$, $P'=\Per\, \yy'$ finds the exact value of $E^*(\vv)$.

\section{Colouring with linked parameters}
In this section we study  how to restrict without loss of generality  the periods $P, P'$ of constant gap sequences  when searching for balanced sequences with the minimal asymptotic critical exponent. 
 Let us recall  that in the whole paper we work without loss of generality with Sturmian sequences $\uu$ having the slope $\theta = \frac{\rho_{\tt a}({\uu})}{\rho_{\tt b}({\uu})}  \in (0,1)$, i.e.,  the letter $\tt b$ is the most frequent letter in ${\uu}$.

\begin{proposition}\label{symmetry} Let  $\yy$ and $\yy'$ be two constant gap sequences and  $\uu \in \{\tt a,  \tt b\}^\N $  be a  Sturmian sequence with  $\tt b$ being  the most frequent letter in ${\uu}$.   Then there exists a  Sturmian sequence $\widetilde{\uu} \in \{\tt a,  \tt b\}^\N $  with  $\tt b$ being  the most frequent letter in $\widetilde{\uu}$ such that  $$E^*\bigl(\colour(\widetilde{\uu},\yy',{\yy})\bigr) =  E^*\bigl(\colour({\uu},\yy,{\yy'})\bigr) \,. $$
\end{proposition}
\begin{proof} 
Let $\theta=[0,a_1,a_2,a_3,\ldots]$ be the slope of $\uu$ and $\frac{p_N}{q_N}$ be  the $N^{th}$ convergent to $\theta$.  Define the slope $\widetilde \theta=[0,b_1,b_2,\ldots]$ of the Sturmian sequence  $\widetilde \uu$ by 
$b_1 =  P =\Per\,{\yy}$  and $b_{N+1} = a_N$  for each $N \in \N$. 
For the slope $\theta$,   we use the notation  $\mathcal{S}(N,m)$ and $\Phi_N$ as introduced in Section~\ref{sec:formulaE*} and $Q_N = p_N+q_N$.  Analogously, for the slope $\widetilde{\theta}$, we 
use the notation $\widetilde{\mathcal{S}}(N,m)$, $\widetilde{\Phi}_N$ and $\widetilde{Q}_N$. 

Using the recurrent relation~\eqref{eq:convergents}, we get $\widetilde{p}_{N+1}=q_N, \quad \widetilde{q}_{N+1}=p_N+q_N P$ and $\widetilde{Q}_{N+1} = Q_N+ q_N P$ for each $N\in \mathbb N$. 
It follows then immediately  for each $N \in \N$  and $m< a_{N+1}  = b_{N+2}$ that 
\begin{enumerate}
\item  $\widetilde{\mathcal{S}}(N+1,m) =  {\mathcal{S}}(N,m)$;
\item $|\frac{\widetilde{Q}_N}{\widetilde{Q}_{N+1}} -\tfrac{Q_{N-1}}{Q_N}| \leq \tfrac{P}{Q_{N}^2} $.
\end{enumerate}

Items 1 and 2  imply that $\lim\limits_{N \to \infty} \bigl( \widetilde{\Phi}_{N+1} - \Phi_N\bigr) = 0$. Hence  by Proposition~\ref{vzorec}  $$E^*\bigl(\colour(\widetilde{\uu},\yy',{\yy})\bigr) = 1+\limsup\limits_{N\to \infty}  \widetilde{\Phi}_{N} =  1+\limsup\limits_{N\to \infty}{\Phi}_N = E^*\bigl(\colour({\uu},\yy,{\yy'})\bigr).$$ 
\end{proof}

In the statement of  Proposition~\ref{symmetry},  the asymptotic critical exponent  cannot be  replaced by the  critical exponent. The following example demonstrates this fact. 

\begin{example}  

Baranwal and Shallit show in \cite{BaSh19} that the minimal critical exponent of $5$-ary balanced sequences equals $\frac{3}{2}$ and it is reached by the sequence $\vv=\colour(\uu,\yy,\yy')$, where $\uu \in \{\tt a,\tt b\}^\N $ is a Sturmian sequence  with slope  $[0,1,\overline{2}]$ and $\tt a$'s are coloured by   $\yy=(\tt 01)^{\omega}$ and $\tt b$'s are coloured by  $\yy'=(\tt 2324)^{\omega}$.

\medskip

Now we  colour $\tt a$'s by $\yy'=(\tt 2324)^{\omega}$ and $\tt b$'s  by  $\yy=(\tt 01)^{\omega}$.    We will explain that for each sequence $\widetilde{\vv}=\colour(\widetilde{\uu},\yy',{\yy})$, where $\widetilde{ \uu}$ is a Sturmian sequence with slope $[0, a_1,a_2,a_3,\ldots ]$, we have $E(\widetilde{\vv})>\frac{3}{2}$.   
\begin{itemize}
    \item If  $a_1 \geq 2$, 
    then by Proposition~\ref{Prop_Standard} and the definition of slope, $D^{a_1}({\tt aa})={\tt b}^{a_1}{\tt a}{\tt b}^{a_1}{\tt a}$ is a factor of $\widetilde{\uu}$. If $a_1 = 1$ and $a_2=1$, then $DG({\tt bba})=\tt babbabba$ is a factor of $\widetilde{\uu}$. In both cases, ${\tt bbabb} \in \mathcal{L}(\widetilde{\uu})$. Hence  by  Corollary~\ref{cor:cyclicshift} the factor  ${\tt 01201} \in  \mathcal{L}(\widetilde{\vv}) $ and thus $E(\widetilde{\vv}) \geq \frac53 > \frac32$.

    \item If $a_1 =1$ and $a_2\geq 2$, then  
    $DG^{a_2}({\tt ab})={\tt ba}({\tt ba})^{a_2}{\tt b}$ is a factor of $\widetilde{\uu}$, in particular,  ${\tt bababab} \in \mathcal{L}(\widetilde{\uu})$. 
    Hence ${\tt 0213021} \in  \mathcal{L}(\widetilde{\vv}) $ and thus $E(\widetilde{\vv}) \geq \frac74 > \frac32$. 
\end{itemize}

\end{example}

\begin{lemma}\label{delitele}
Let $\vv=\colour(\uu,\yy,{\yy}')$ and $\widehat{\vv}=\colour(\uu,\hat\yy,{\hat\yy}')$. If $\Per\,\hat\yy$  is divisible by $\Per\,\yy$  and $\Per\,{\hat\yy}'$ is divisible by $\Per\,{\yy}'$, then $E^*(\vv)\geq E^*(\widehat{\vv})$.

\end{lemma}

\begin{proof}
Let us denote by ${\mathcal S}(N,m)$ the set corresponding to $\vv$ defined in Section~\ref{sec:formulaE*} and similarly $\hat{{\mathcal S}}(N,m)$ for $\widehat{\vv}$. Since we colour the same Sturmian sequence $\uu$, we have $\hat{{\mathcal S}}(N,m) \subset {\mathcal S}(N,m)$. Applying Proposition~\ref{vzorec} we obtain $E^*(\vv)\geq E^*(\widehat{\vv})$.
\end{proof}

\section{Equivalence on unimodular matrices}
By Proposition \ref{prop:bounds}, the asymptotic critical exponent depends only on the periods of constant gap sequences  $\yy$ and $\yy'$, not on their structure.   
In the sequel we  use for a fixed pair    $P=\Per\, \yy\,$ and $ P'=\Per\, \yy'\,$  the notation $H,L,Y,Y'$  introduced by the following relations:  
\begin{equation}\label{notataionH}
H = \gcd(P,\ P'\, ), \ \  P = HY\,,  \  \  P' = HY' \  \  \text{and } \ \ L = \lcm(P, P'\, ).
\end{equation}
Obviously, $Y$ and $Y'$ are coprime. We always consider $L>1$ since for $\Per\, \yy=\Per\, \yy'=1$, the sequences $\vv=\colour(\uu, \yy, \yy')$ and $\uu$ are the same and the minimal asymptotic critical exponent for Sturmian sequences is known.

In this section we will study the form of the set 
$$\mathcal{S}_1(N,m)=\bigl\{ \left( \begin{smallmatrix}  \ell \\ k \end{smallmatrix} \right) : A_N\left( \begin{smallmatrix} 1 & 0 \\  m & 1 \end{smallmatrix} \right)  \left( \begin{smallmatrix}  \ell \\ k \end{smallmatrix} \right) = \left( \begin{smallmatrix}  0\\ 0 \end{smallmatrix} \right) \mod \left( \begin{smallmatrix}  P \\ P' \end{smallmatrix} \right) \bigr\}\,,$$
which plays an essential role in the definition of $\Phi_N$ (see Proposition~\ref{vzorec}), and consequently in the computation of the asymptotic critical exponent of balanced sequences.  
Note  that the determinant of $ A_N\left( \begin{smallmatrix} 1 & 0 \\  m & 1 \end{smallmatrix} \right)$ equals $\pm1$, i.e., the matrix is unimodular. A solution  $\left( \begin{smallmatrix}  \ell \\ k \end{smallmatrix} \right) $ depends only on  entries of the matrix counted $\bmod \,P$ in the first row and $\bmod \, P'$ in the second row.  Hence it is possible to group matrices into classes of the same behaviour with respect to the form of $\mathcal{S}_1(N,m)$.  The following lemma prepares 
such grouping.

\begin{lemma}\label{le:LambdaKappa} Let $A \in \mathbb Z^{2\times 2}$ be unimodular and   $\ell, k \in \mathbb{Z}$. 
Then 
 $A\left( \begin{smallmatrix}  \ell \\ k \end{smallmatrix} \right)=
 \left( \begin{smallmatrix}0 \\ 0 \end{smallmatrix} \right) \ {\rm mod} \ 
 \left( \begin{smallmatrix}  P \\ P' \end{smallmatrix} \right)$ if and only if there exist $\lambda, \kappa \in \mathbb{Z}$ such that $$
\left( \begin{smallmatrix}  \ell \\ k \end{smallmatrix} \right)= H\left( \begin{smallmatrix}  \lambda  \\ \kappa \end{smallmatrix} \right) \ \ \ \text{and} \  \ \ 
 A\left( \begin{smallmatrix}  \lambda  \\ \kappa  \end{smallmatrix} \right)=\left( \begin{smallmatrix} 0 \\ 0  \end{smallmatrix} \right)  \ {\rm mod}  \left( \begin{smallmatrix}  Y \\  Y' \end{smallmatrix} \right) . 
 $$ 
\end{lemma}
\begin{proof} $(\Rightarrow)$
We have   $A\left( \begin{smallmatrix}  \ell \\ k \end{smallmatrix} \right)=\left( \begin{smallmatrix} a\, P\\ b\, P'\end{smallmatrix} \right)=H\left( \begin{smallmatrix} a  Y\hspace{0.1cm} \\ b  Y' \end{smallmatrix} \right)$ for some $a,b \in \mathbb Z$. Since $A$ is unimodular, it is invertible in $\mathbb Z$, hence $$\left( \begin{smallmatrix}  \lambda \\ \kappa \end{smallmatrix} \right) :=\tfrac{1}{H}\left( \begin{smallmatrix}  \ell \\ k \end{smallmatrix} \right)=A^{-1} \left( \begin{smallmatrix} a  Y\hspace{0.1cm} \\ b  Y' \end{smallmatrix} \right) \in \mathbb{Z}^2.$$
Moreover, $A\left( \begin{smallmatrix}  \ell \\ k \end{smallmatrix} \right)=H\left( \begin{smallmatrix} a  Y\hspace{0.1cm} \\ b  Y' \end{smallmatrix} \right)$ implies 
$A\left( \begin{smallmatrix}  \lambda \\ \kappa \end{smallmatrix} \right)=\left( \begin{smallmatrix} a Y\hspace{0.1cm} \\ b  Y' \end{smallmatrix} \right) =\left( \begin{smallmatrix} 0 \\ 0  \end{smallmatrix} \right)  \ {\rm mod}  \left( \begin{smallmatrix}  Y \\  Y' \end{smallmatrix} \right) . $

$(\Leftarrow)$ The reasoning is analogous. 
\end{proof}

\begin{remark}\label{name} To solve the equation    $A\left( \begin{smallmatrix}  \lambda \\ \kappa \end{smallmatrix} \right)=
 \left( \begin{smallmatrix}0 \\ 0 \end{smallmatrix} \right) \ {\rm mod} \ 
 \left( \begin{smallmatrix}  Y \\ Y' \end{smallmatrix} \right)$, where $A$ is a unimodular matrix,  we just  need to know  the remainders of  division  by $Y$, resp. $Y'$, of the first row, resp. of the second row of $A$. Therefore, we consider instead of  a  matrix $ A=\left( \begin{smallmatrix}  a_{11}&a_{12} \\ a_{21}&a_{22} \end{smallmatrix} \right)$ the so-called   
$$ \text{$(Y,Y')$-name of $A$ defined as   } \quad M_A= \left( \begin{smallmatrix}  a_{11}\bmod Y\ & \ a_{12} \bmod Y \\   a_{21}\bmod Y'\ & \ a_{22} \bmod Y' \end{smallmatrix} \right). 
$$ 
Obviously, $A\left( \begin{smallmatrix}  \lambda  \\ \kappa  \end{smallmatrix} \right)=\left( \begin{smallmatrix} 0 \\ 0  \end{smallmatrix} \right)  \ {\rm mod}  \left( \begin{smallmatrix}  Y \\  Y' \end{smallmatrix} \right)$ if and only if $M_A\left( \begin{smallmatrix}  \lambda  \\ \kappa  \end{smallmatrix} \right)=\left( \begin{smallmatrix} 0 \\ 0  \end{smallmatrix} \right)  \ {\rm mod}  \left( \begin{smallmatrix}  Y \\  Y' \end{smallmatrix} \right)$.   Hence, matrices with the same  $(Y,Y')$-name have the same  solutions $\lambda, \kappa$.
 By Lemma \ref{le:LambdaKappa}, the values $H$, $m$ and the  $(Y,Y')$-name  of $A_N$  capture all pieces of information we need to decide whether $ \left( \begin{smallmatrix}  \ell \\ k \end{smallmatrix} \right)$ belongs to the set $\mathcal{S}_1(N,m)$.  
\end{remark}

  Nevertheless,  matrices with distinct $(Y,Y')$-names can have the same  solutions as well. The following example illustrates it. 

\begin{example} Let $Y=3$ and $Y'=4$. Consider the unimodular matrices $ A=\left( \begin{smallmatrix}  5& 31 \\ {21}&130 \end{smallmatrix} \right)$ and  $ B=\left( \begin{smallmatrix}  1& 11 \\ 3&34 \end{smallmatrix} \right)$. Their $(3,4)$-names are $M_A =\left( \begin{smallmatrix}  2& 1 \\ {1}&2 \end{smallmatrix} \right)$  and  $M_B =\left( \begin{smallmatrix}  1& 2 \\ {3}&2 \end{smallmatrix} \right)$.  By the previous remark,  a pair $\lambda,\kappa$  solves 
$A\left( \begin{smallmatrix}  \lambda  \\ \kappa  \end{smallmatrix} \right)=\left( \begin{smallmatrix} 0 \\ 0  \end{smallmatrix} \right)  \ {\rm mod}  \left( \begin{smallmatrix}  3 \\  4 \end{smallmatrix} \right)$  if and only if $$ \begin{smallmatrix}  2\lambda+\kappa= 0 \mod 3\\ \lambda +2\kappa = 0 \mod 4\end{smallmatrix} \quad \Longleftrightarrow  \quad  \begin{smallmatrix}  2(2\lambda+\kappa)= 0 \mod 3\\ 3(\lambda +2\kappa) = 0 \mod 4\end{smallmatrix} \quad \Longleftrightarrow \quad \begin{smallmatrix}  \lambda+2\kappa= 0 \mod 3\\ 3\lambda +2\kappa = 0 \mod 4\end{smallmatrix}\,, $$
which is equivalent to  $B\left( \begin{smallmatrix}  \lambda  \\ \kappa  \end{smallmatrix} \right)=\left( \begin{smallmatrix} 0 \\ 0  \end{smallmatrix} \right)  \ {\rm mod}  \left( \begin{smallmatrix}  3 \\  4 \end{smallmatrix} \right)$.   
\end{example}

To group unimodular matrices into classes with the same pairs $\lambda, \kappa$ of  solutions, we introduce an equivalence. 

\begin{definition}\label{def:equivalence} 
Let $A$ and $ B$ be unimodular matrices in $\mathbb Z^{2\times 2}$.  
We say that $A$ is {\em equivalent} to $B$, and write $A \equiv B$, if there exist $c \in \mathbb{Z}$ coprime with $Y$ and $c' \in \mathbb{Z}$ coprime with $Y'$ such that 
\begin{equation}\label{stejne}
 \left( \begin{smallmatrix} c & 0 \\  0 & c' \end{smallmatrix} \right) A = B \  \ {\rm mod}  \left( \begin{smallmatrix}  Y \\  Y' \end{smallmatrix} \right)\,.
\end{equation}
 The equivalence  class containing a matrix $A$ will be denoted $[A]_\equiv$.
\end{definition}

\begin{remark}
The relation $\equiv$ is an equivalence. Evidently $A \equiv A$. For a~fixed $z \in \mathbb Z$, the set of elements coprime with $z$ is closed under inverse modulo $z$ and multiplication modulo $z$, which implies that $\equiv$ is symmetric and transitive. 
Obviously, if  $A$ and $B$ satisfy \eqref{stejne}, then the  $(Y,Y')$-names  $M_A$ and $M_B$ satisfy
 $\left( \begin{smallmatrix} c & 0 \\  0 & c' \end{smallmatrix} \right) M_A = M_B \  \ {\rm mod}  \left( \begin{smallmatrix}  Y \\  Y' \end{smallmatrix} \right)$.
\end{remark}

\begin{lemma}\label{lem:equivalent_solutions1}
Let  $k, \ell \in \mathbb N$ and let $A$ and $B$ be  unimodular matrices in $\mathbb{Z}^{2\times 2}$ such that   $A \equiv  B$. 
Then
\begin{enumerate}
\item $A\left( \begin{smallmatrix}  \ell \\ k \end{smallmatrix} \right)=\left( \begin{smallmatrix} 0 \\ 0 \end{smallmatrix} \right) \ {\rm mod}\  \left( \begin{smallmatrix}  P \\  P' \end{smallmatrix} \right) 
\text{\ \  if and only if \ \  }B\left( \begin{smallmatrix}  \ell \\ k \end{smallmatrix} \right)=\left( \begin{smallmatrix} 0 \\ 0 \end{smallmatrix} \right) \ {\rm mod}\  \left( \begin{smallmatrix}  P \\  P' \end{smallmatrix} \right)  .$
    \item $AC \equiv  BC$  \ \ for any unimodular matrix $C \in \mathbb{Z}^{2\times 2}$.
   \end{enumerate}

\end{lemma}
\begin{proof} 
Let $A$ and $B$ satisfy \eqref{stejne} and  $\lambda, \kappa \in \N$.    Then   \begin{equation}\label{druhy}\left( \begin{smallmatrix} c & 0 \\  0 & c' \end{smallmatrix} \right) A \left( \begin{smallmatrix}  \lambda \\ \kappa \end{smallmatrix} \right)= B \left( \begin{smallmatrix}  \lambda \\ \kappa \end{smallmatrix} \right)\  \ {\rm mod}  \left( \begin{smallmatrix}  Y \\  Y' \end{smallmatrix} \right)\,.
\end{equation}
Let us point out a simple fact: If $c$ is coprime with $Y$,   then  $cx = 0 \mod Y$ if and only if $x=0\mod Y$ for every $x \in \mathbb{Z}$ and analogously for the coprime values $c'$ and $Y'$.   Hence 
\begin{equation}\label{treti}\left( \begin{smallmatrix} c & 0 \\  0 & c' \end{smallmatrix} \right) A\left( \begin{smallmatrix}  \lambda \\ \kappa \end{smallmatrix} \right)=\left( \begin{smallmatrix} 0 \\ 0 \end{smallmatrix} \right) \ {\rm mod}\  \left( \begin{smallmatrix}  Y \\  Y' \end{smallmatrix} \right) 
\text{\ \  if and only if \ \  }A\left( \begin{smallmatrix}  \lambda \\ \kappa \end{smallmatrix} \right)=\left( \begin{smallmatrix} 0 \\ 0 \end{smallmatrix} \right) \ {\rm mod}\  \left( \begin{smallmatrix}  Y \\  Y' \end{smallmatrix} \right)  .\end{equation}
Equations \eqref{druhy} and \eqref{treti} imply 
\begin{equation}\label{prvni}A\left( \begin{smallmatrix}  \lambda \\ \kappa \end{smallmatrix} \right)=\left( \begin{smallmatrix} 0 \\ 0 \end{smallmatrix} \right) \ {\rm mod}\  \left( \begin{smallmatrix}  Y \\  Y' \end{smallmatrix} \right) 
\text{\ \  if and only if \ \  }B\left( \begin{smallmatrix}  \lambda \\ \kappa \end{smallmatrix} \right)=\left( \begin{smallmatrix} 0 \\ 0 \end{smallmatrix} \right) \ {\rm mod}\  \left( \begin{smallmatrix}  Y \\  Y' \end{smallmatrix} \right). \end{equation}
 Lemma \ref{le:LambdaKappa} together with  Equation \eqref{prvni} prove Item 1.  

Item 2 is a direct consequence of  Equation \eqref{druhy} in which $\left( \begin{smallmatrix}  \lambda \\ \kappa \end{smallmatrix} \right)$ represents  the first column of the matrix $C$  and then its second column.
\end{proof}

%%%%%%%%%%%%%%%%%%%%%%%%%%%%%%%%%%%%%%%%%%%%%%%%%%%%%%%%%%

\section{A lower bound on the asymptotic critical exponent for fixed periods of constant gap sequences}

We associate to $\theta = [0,a_1,a_2,a_3, \ldots]$ a sequence of classes of equivalent matrices $\bigl([A_N]_\equiv\bigr)$ with representatives  $A_N=\left( \begin{smallmatrix} p_{N-1}  & p_N\\  q_{N-1} & q_{N} \end{smallmatrix} \right)$  and a sequence of $(\delta_N)$ with $\delta_N = [a_{N+1}, a_{N+2}, \ldots ]$. 
Let us stress that $A_N$ depends only of the first $N$ coefficients of the continued fraction of $\theta$, whereas $\delta_N$ depends only on the remaining coefficients of $\theta$.  Representatives of two consecutive classes satisfy
\begin{equation}\label{nasledovnik}A_{N+1} = A_N\left( \begin{smallmatrix}  0&1\\ 1&a_{N+1} \end{smallmatrix} \right)\,.\end{equation}

\begin{definition}\label{Forcing} Let $A \in \mathbb{Z}^{2\times 2}$ be unimodular and $\beta >0$.
We say that  $\delta >1$  is $(1+\beta)$-{\em forcing} for the class $[A]_\equiv$ if there exist $m, k, \ell \in \N,  \ell +k >0$, such that   
\begin{description}

\item[$\mathfrak{P}1$]   $A\begin{pmatrix} 1 & 0 \\  m & 1 \end{pmatrix} \begin{pmatrix}  \ell\\ k \end{pmatrix} =\begin{pmatrix} 0  \\ 0  \end{pmatrix} \mod \begin{pmatrix}  P  \\ P' \end{pmatrix}\,;$ 

\medskip

\item [$\mathfrak{P}2$]   $m+1<\delta$ \ \ and  \ \  $|\ell(\delta-m) -k| < \delta -m+1\,;$

\medskip

\item [$\mathfrak{P}3$]  

\begin{itemize}
\item if $k=\ell $, then $ \frac{1}{k} > \beta$;
\item if $k> \ell $, then $ \frac{1+m}{k+m\ell} \geq \beta$;
\item if $ k < \ell $, then $\frac{2+m}{k +(m+1)\ell} \geq \beta$.
\end{itemize}

\end{description}
The set of $(1+\beta)$-forcing $\delta$'s  for the  class $[A]_{\equiv}$ is denoted $\mathcal{F}(\beta, A)$.

\end{definition}

Note that the definition is correct because it does not depend on the choice of the representative $A$ from the class of equivalence. Indeed, only $\mathfrak{P}1$ depends on $A$ and by Lemma~\ref{lem:equivalent_solutions1} the solutions  $\ell, k$ are the same for each  matrix in $[A]_\equiv$.

\begin{theorem}\label{BetaForcing}
Let $\vv=\colour(\uu,\yy,\yy')$, where $\uu$  is a Sturmian sequence  with slope $\theta = [0, a_1,a_2, a_3, \ldots]$, and let $\beta$ be a fixed positive number.

Assume that there exist infinitely many $N\in \N$ such that $\delta_N$ is $(1+\beta)$-forcing for the class $[A_N]_\equiv$. 
Then $E^*(\vv) > 1+ \beta$.
\end{theorem}

\begin{proof} 
First assume that the sequence $(a_n)$ of coefficients in the continued fraction expansion of $\theta$ is unbounded, then by Proposition~\ref{prop:bounds} $E^*(\vv)=+\infty$. In the sequel, assume that $(a_n)$ is bounded, say  by $K \in \N$.

Let  $N\in \N$  be such that  $ \delta_N$ is $(1+\beta)$-forcing  for the class $[A_N]_\equiv$.   We point out  that 
$\mathfrak{P}1$ and $\mathfrak{P}2$ of Definition \ref{Forcing} together  mean that $m< a_{N+1}$  and  $\left( \begin{matrix}  \ell \\ k \end{matrix} \right)$ belongs to $\mathcal{S}(N,m)$ -- the set used in the definition of $\Phi_N$ in Proposition \ref{vzorec}. Hence $\Phi_N$  can be rewritten as
\begin{equation}\label{FI}
\Phi_N = \max\Bigl\{\tfrac{1+m +x_N}{k + \ell m +\ell x_N}: m,k,\ell  \ \text{satisfy } \mathfrak{P}1\text{ and }  \mathfrak{P}2  \text{ for }  A=A_N \text{ and } \delta = \delta_N\Bigr\}, 
\end{equation}
where we abbreviate notation putting  $x_N = \frac{Q_{N-1}}{Q_N} \in [0,1]$. Using $\mathfrak{P}3$, we get 
\begin{equation}\label{funkceF}  \tfrac{1+m +x_N}{k + \ell m +\ell x_N} \geq  \min\Bigl\{\tfrac{1+m + x}{k+\ell m + \ell x}: x \in [0,1]\Bigr\} =\left\{\begin{array}{cl}
\frac1k > \beta & \text{if}\ k = \ell\,; \\      &\\  \frac{1+m}{k+\ell m } \geq \beta 
     &  \text{if}\ k > \ell\,; \\
     &\\
 \frac{2+m}{k+\ell (m+1) } \geq  \beta     & \text{if}\ k\ < \ell\,.
\end{array} \right. 
\end{equation} 
Hence  $\Phi_N \geq \beta$ for infinitely many $N$.  It  implies $E^*(\vv) = 1+\limsup\limits_{N\to \infty} \Phi_N \geq 1+\beta$. 

To prove the  strict inequality  $E^*(\vv) > 1+\beta$, we need to show more, namely,  that there exists a positive number, say $\mu >0$,  such that $\Phi_N\geq \mu +\beta$ for each $N$, where $\delta_N$ is $(1+\beta)$-forcing.  Existence of such $\mu$  follows from the  three facts:

\medskip

1) $x_N \in \bigl[\frac{1}{K+1}, \frac{K+1}{K+2}\bigr] \subset [0,1]$. Indeed, the recurrence relation $Q_{N} = a_{N}Q_{N-1} +Q_{N-2}$ and the inequalities  $1\leq a_n\leq K$  imply on one hand 
$$x_N=\frac{Q_{N-1}}{Q_N}=\frac{Q_{N-1}}{a_NQ_{N-1}+Q_{N-2}}\geq \frac{Q_{N-1}}{KQ_{N-1}+Q_{N-1}}=\frac{1}{K+1}$$
and on the other hand
$$x_N=\frac{1}{a_N+\frac{Q_{N-2}}{Q_{N-1}}}\leq \frac{1}{1+\frac{1}{K+1}}=\frac{K+1}{K+2}\,.$$

2) If $k\neq \ell$, the function $f_{m,k,\ell}(x) = \tfrac{1+m +x}{k + \ell m +\ell x} $ is strictly monotonous and thus the minimum of $f_{m,k,\ell}$ on the interval  $[\frac{1}{K+1}, \frac{K+1}{K+2}]$ is strictly bigger than the minimum on $[0,1]$. If $k=\ell$, the strict inequality is required by   $\mathfrak{P}3$. 

\medskip

3) There are only finitely many triplets  $m,k,\ell \in \N $ satisfying   $\mathfrak{P}2$, $\mathfrak{P}3$ and  $m <K$. 
\end{proof}

\begin{example} \label{ex:4lettersperiod13} Using the previous proposition we  show that
  $P=\Per\,\yy=1$ and  $P'=\Per\,\yy'=3$ implies  $E^*(\vv)\geq 2$ for every   $\vv = \colour(\uu, \yy,\yy')$.  In particular, we show that for any $\beta \in (0,1)$ every sequence of $(\delta_N)$  contains  infinitely many  $(1+\beta)$-forcing $\delta_N$.  
  
There are four classes of equivalent matrices with $(Y,Y')$-names:

\medskip
\centerline{$M_1=\left( \begin{smallmatrix}  0&0 \\ 1&1 \end{smallmatrix} \right), M_2=\left( \begin{smallmatrix}  0&0 \\ 1&2 \end{smallmatrix} \right), M_3=\left( \begin{smallmatrix}  0&0 \\ 1&0 \end{smallmatrix} \right), M_4=\left( \begin{smallmatrix}  0&0 \\ 0&1 \end{smallmatrix} \right)$. }
\medskip
\noindent If $[A]_\equiv$ has the name 
\begin{itemize}
    \item $M_1$,  then every $\delta > 2$ is $(1+\beta)$-forcing for $[A]_\equiv$, as  $\delta$ and $A$ satisfy Properties  $\mathfrak{P}1$, $\mathfrak{P}2$ and $\mathfrak{P}3$ with the triplet $m=k=\ell =1$.  
      \item $M_2$,  then every $\delta > 1$ is $(1+\beta)$-forcing for $[A]_\equiv$, the corresponding  triplet is  $m=0$, $k=\ell =1$. 
     \item $M_3$,  then every $\delta > 1$ is $2$-forcing for $[A]_\equiv$, the corresponding  triplet is  $m=0$, $k=1, \ell = 0$.
       \item $M_4$,  then every $\delta > 1$ is $2$-forcing for $[A]_\equiv$, the corresponding  triplet is  $m=0$, $k=0, \ell = 1$.
\end{itemize}
Therefore, if $E^*(\vv)$ is smaller than  $2$ for some $\vv$,  then necessarily  $\delta_N< 2$ and  $M_1$ is the name of  $[A_N]_\equiv$ for all $N > N_0$.  In particular, $a_{N+1} = \lfloor\delta_N\rfloor = 1$ for all $N>N_0$. 
However, if $[A_N]_\equiv$ has the name $M_1$, then the class $[A_{N+1}]_\equiv$ containing the matrix 
$A_{N+1} = A_N\left( \begin{smallmatrix}  0&1 \\ 1&1 \end{smallmatrix} \right)$ has the name $M_3$ -- a contradiction. 

\end{example}

\begin{example} \label{ex:4lettersperiod22}
Now we apply Theorem \ref{BetaForcing} to balanced sequences obtained by  colouring  with two constant gap sequences both having the period 2, i.e.,    $P=P'=2$.   Using notation \eqref{notataionH}, we have $H=2, Y=1, Y'=1$. According to Definition \ref{def:equivalence} all integer unimodular matrices belong to the same class of equivalence. By Lemma \ref{le:LambdaKappa}, the triplet $m=1$, $k=2$ and $\ell = 0$ satisfies Property $\mathfrak{P}1$ for any unimodular  matrix $A$. 
   
If we  fix $\beta =1$, then every $\delta > 2$  with the same  triplet satisfies $\mathfrak{P}2$  and $\mathfrak{P}3$.  In other words, $\delta >2$ is $2$-forcing.   Therefore, the only candidates for $\vv$  with $E^*(\vv)<2$ are  colourings of  Sturmian sequences with slope $\theta = [0,w, \overline{1}]$, where $w$ is any finite preperiod.     
For every such  $\theta$, the formula \eqref{Eq:Phi} in Proposition~\ref{vzorec} gives $\Phi_N=\frac{1+\frac{Q_{N-1}}{Q_N}}{2}$ for sufficiently large $N$, where $(Q_N)$ fulfils the recurrence relation   $Q_{N+1}=Q_N+Q_{N-1}$. Consequently,   
$E^*(\vv)= 1+\lim \Phi_N=1+\frac{\sqrt{5}+1}{4}\doteq 1.809$ is the minimum value of the asymptotic critical exponent for $P=P'=2$ and it is attained if and only if $\vv$ is a colouring of a Sturmian sequence with slope $\theta = [0, w, \overline{1}] $ for a finite preperiod $w$.  

\end{example}

\section{Admissible tails of continued fractions}
 In the previous section, we have associated with a continued fraction $\theta=[0,a_1, a_2, a_3, \ldots]$ 
a sequence of classes of equivalent matrices $\bigl([A_N]_\equiv\bigr)$. Since the number of classes is finite, there exists $N_0$ such that  any  class of equivalence either   occurs  in  $\bigl([A_N]_\equiv\bigr)_{N > N_0}$ infinitely many times or  does not occur at all in it. 
To find a   balanced sequence $\vv = \colour(\uu, \yy, \yy')$ with $E^*(\vv) \leq 1+\beta$, we should at least guarantee that 
$\delta_N$ is not $(1+\beta)$-forcing  for the class $[A_N]_\equiv$ for each $N>N_0$. 
Formally, $\delta_ N\notin \mathcal{F}( \beta, A_N)$.

\begin{remark}\label{obs:simplification}   In  Definition \ref{Forcing} only Property  $\mathfrak{P}2$  depends on $\delta$. Therefore,  for each triplet $m,k,\ell$ satisfying  $\mathfrak{P}1$  and $\mathfrak{P}3$,  we add to $\mathcal{F}(\beta, A)$ the interval of $\delta$'s satisfying $\mathfrak{P}2$, i.e., the interval  
$$ 
\begin{cases*}
 \  (m+ k-1, +\infty)\cap (m+1, +\infty)    & \quad if $\ell = 0$;\\   \ (m+ \frac{k-1}{2}, +\infty) \cap (m+1, +\infty) & \quad  if $\ell = 1$;\\
   \  (m+ \frac{k-1}{\ell+1}, m+ \frac{k+1}{\ell-1} )\cap (m+1, +\infty)  & \quad if $\ell \geq 2$.
\end{cases*}
$$
\end{remark}
The set $\mathcal{F}(\beta, A)$ is a union of several  open  intervals. Boundaries of these intervals are rational.  Rational $\delta$'s  have finite continued fractions and do not occur as tails of the continued fraction expansion of slopes of Sturmian sequences.

\begin{definition}\label{def:admissible_intervals} Let $\beta> 0$ and $A \in \mathbb{Z}^{2\times 2}$ be unimodular.   We denote 
$$\mathcal{D}(\beta, A)=  \{\delta > 1: \delta \text{ is NOT in the closure of }\  \mathcal{F}(\beta, A)  \}. $$
\end{definition}

Using the notation of $\mathcal{D}(\beta, A)$,   Theorem \ref{BetaForcing} can be rephrased  as the following corollary.

\begin{corollary}\label{cesta} Let $\theta = [0,a_1,a_2,a_3,\ldots]$ be the slope of a Sturmian sequence $\uu$ and $\beta >0$.  
If  $E^*(\vv) \leq 1+\beta$,  then there exists $N_0$ such that  for every $N>N_0$ the set $\mathcal{D}(\beta, A_N)$ contains  $\delta_N$.
\end{corollary}

\begin{lemma}\label{omez} Let $\beta >0$  and $L=\lcm(P,P'\,) >1$. Then the  set 
$\mathcal{D}(\beta, A)$ is a subset of $ (1, \lceil L(1+\beta)\rceil -2 ) $.  In particular  $\mathcal{D}(\beta, A)$ is bounded for each  equivalence class $[A]_\equiv$. 
\end{lemma}
\begin{proof}
Note that the pair $ k=L$ and $\ell=0$ fulfills $\mathfrak{P}1$ of Definition \ref{Forcing} independently on $A$  and $m\in \N$. If  we take  $m=\lceil L\beta\rceil -1$, then   $\frac{1+m}{L} \geq \beta$ and  $\mathfrak{P}3$ of Definition \ref{Forcing} is satisfied as well.  By Remark \ref{obs:simplification},  the set $(m +k -1, +\infty) = ( \lceil L(1+\beta)\rceil -2, +\infty )$ belongs to   $\mathcal{F}(\beta, A)$.    
\end{proof}

   To compute easily $\mathcal{D}(\beta,A)$,   we  collect several practical observations on triplets $m,k,\ell \in \mathbb N$ that may influence the form of  $\mathcal{D}(\beta,A)$.  We will apply these rules in examples worked out in hand. They follow immediately from Definitions~\ref{Forcing} and~\ref{def:admissible_intervals}, from Remark~\ref{obs:simplification}  and  Lemmas~\ref{le:LambdaKappa} and~\ref{omez}.  
\begin{remark}\label{rem:ParikhAlways}  Let $A$ be a unimodular matrix, $M$ denotes its $(Y,Y')$-name.  The following statements hold true. 

\begin{enumerate} 

\item  The number of triplets $m,k, \ell$ that may affect $\mathcal{D}(\beta,A)$ is bounded. 
 In particular, by Definition~\ref{def:admissible_intervals} and Lemma~\ref{omez} it holds $(1,m+1) \subset \mathcal{D}(\beta,A) \subset (1, \lceil L(1+\beta)\rceil -2)$. It together with  Property $\mathfrak{P}3$ forces $m,k , \ell$ to satisfy  $0\leq m < \lceil L(1+\beta)\rceil -2 \ \text{and} \   k+\ell m \leq \frac{1}{\beta}(m+2).$ The number of such triplets  is finite.

 Each triplet adds to the set $\mathcal{F}(\beta, A)$ an interval as described in Remark \ref{obs:simplification}.  In fact, only  few of the triplets add really new elements to $\mathcal{F}(\beta,A)$, or equivalently, erase some elements of  $\mathcal{D}(\beta,A)$.    Hence, when  we present in examples the set $\mathcal{D}(\beta,A)$, we  list only the   triplets which  determine the set. It means that no other triplet  reduces the set $\mathcal{D}(\beta,A)$ any more.

\medskip

\item $\mathfrak{P}2$ is satisfied whenever   $m +1 <\delta$ and   
 $$\left( \begin{smallmatrix}  \ell \\ k \end{smallmatrix} \right)\in \{\left( \begin{smallmatrix}  1 \\ 0 \end{smallmatrix} \right), \left( \begin{smallmatrix}  0 \\ 1 \end{smallmatrix} \right), \left( \begin{smallmatrix}  1 \\ 1 \end{smallmatrix} \right), \left( \begin{smallmatrix}  0 \\ 2 \end{smallmatrix} \right), \left( \begin{smallmatrix}  1 \\ 2 \end{smallmatrix} \right), \left( \begin{smallmatrix}  1 \\ 3 \end{smallmatrix} \right)\}.$$
 \item   $\left( \begin{smallmatrix}  \ell \\ k \end{smallmatrix} \right)$ with $\ell \geq k+2$ does not  fulfil   $\mathfrak{P}2$ for any $m \in \mathbb N$ and $\delta >m+1$.

\medskip

 \item If $H\geq 2$, then only  $k\geq  \ell$ may satisfy $\mathfrak{P}1$ and $\mathfrak{P}2$ for some $m \in \mathbb N$ and $\delta >m+1$. 

\medskip

\item Let  $\beta < 1$ and  $\mathfrak{P}1$ be  fulfilled  for some $m \in \mathbb N$ and $\left( \begin{smallmatrix}  \ell \\ k \end{smallmatrix} \right)\in \{\left( \begin{smallmatrix}  1 \\ 0 \end{smallmatrix} \right), \left( \begin{smallmatrix}  0 \\ 1 \end{smallmatrix} \right), \left( \begin{smallmatrix}  1 \\ 1 \end{smallmatrix} \right)\}$. Then $\mathfrak{P}3$ is fulfilled, too.  Moreover,   $\mathfrak{P}2$  is satisfied for all $\delta>m+1$.

\medskip

Consequently, 

\medskip

\centerline{\quad $\mathcal{D}(\beta, A)=\emptyset$ for $m=0$ \ \ \ and  \ \ \  $\mathcal{D}(\beta, A)\subset (1,m+1)$ for $m\geq 1$.} 

\medskip
 
\item
If a triplet $m,k,\ell$  satisfies Properties  $\mathfrak{P}1$ and   $\mathfrak{P}2$,  then the triplet   $m,Tk, T\ell$ with $T\in \N, T\geq 2$,  need not be taken into account when constructing  $\mathcal{D}(\beta,A)$. 

Indeed, if $m,Tk, T\ell$   satisfies   $\mathfrak{P}1$ and   $\mathfrak{P}2$, then
$$\begin{array}{l}
(m+ Tk-1, +\infty) \subset (m+ k-1, +\infty)\\[0.3em]
(m+ \frac{Tk-1}{2}, +\infty) \subset (m+ \frac{k-1}{2}, +\infty)\\[0.3em]
(m+ \frac{Tk-1}{T\ell+1}, m+ \frac{Tk+1}{T\ell-1} )\subset (m+ \frac{k-1}{\ell+1}, m+ \frac{k+1}{\ell-1} )\ \text{for} \ \ell \geq 2.
\end{array}$$
The statement follows by  Remark \ref{obs:simplification}. 
\end{enumerate}
\end{remark}

\begin{example}\label{Ex: 24}
Let    $P=2$, $P'=4$ and $ \beta=\frac{1}{2}$.
In this case $H=2$, $L=4$, $Y=1$, $Y'=2$ and  there are three classes of equivalent matrices with $(Y, Y')$-names: 
$$M_1=\left( \begin{smallmatrix} 0& 0\\ 1&0 \end{smallmatrix} \right), M_2=\left( \begin{smallmatrix} 0& 0\\ 0&1 \end{smallmatrix} \right), M_3=\left( \begin{smallmatrix} 0& 0\\ 1&1 \end{smallmatrix} \right).$$
To  find  $\mathcal D(\beta, M_i)$ we will apply Remark \ref{rem:ParikhAlways}.

\begin{itemize}
\item $M_1=\left( \begin{smallmatrix} 0& 0\\ 1&0 \end{smallmatrix} \right)$\\
Consider   the triplet $m=0, k=2, \ell=0$.  
$\mathfrak{P}1$ is satisfied by Lemma~\ref{le:LambdaKappa}.  $\mathfrak{P}2$ is satisfied  for $\delta > 1$,  by Item 2.  
As $2=k >\ell=0$, $\mathfrak{P}3$ says  $\frac12=\frac{1+m}{k+m\ell}\geq  \beta=\frac12$. 
 
Hence, any $\delta > 1$ is $(1+\frac12)$-forcing  and thus   $\mathcal D(\beta, M_1) = \emptyset$. 

\item $M_2=\left( \begin{smallmatrix} 0& 0\\ 0&1 \end{smallmatrix} \right)$\\
Due to Lemma~\ref{le:LambdaKappa} we consider only triplets with  $k=2\kappa$ , $\ell = 2 \lambda$, where  $M_2\left( \begin{smallmatrix}  \lambda \\ \kappa \end{smallmatrix} \right) = \left( \begin{smallmatrix}  0 \\ 0 \end{smallmatrix} \right) \bmod \left( \begin{smallmatrix}  1 \\ 2 \end{smallmatrix} \right)$  forces  $\kappa$ to be even. Item 4 gives the restriction $ \kappa\geq \lambda$. Hence by Items 1 and 6, we need to work only with $\left( \begin{smallmatrix}  \ell \\ k \end{smallmatrix} \right)\in \{\left( \begin{smallmatrix}  0 \\ 4 \end{smallmatrix} \right), \left( \begin{smallmatrix}  2 \\ 4 \end{smallmatrix} \right), \left( \begin{smallmatrix}  4 \\ 4 \end{smallmatrix} \right) \}$ and $m < 4$. The inequality $\frac{1+m}{k+m\ell}\geq \frac12$ (in case $k>\ell$) and $\frac1k > \frac12$  (in case $k=\ell$) required by $\mathfrak{P}3$ is fulfilled only if $k=4, \ell =0$ and $1 \leq m<4$. 

By Remark~\ref{obs:simplification}, the set of $(1+\frac12)$-forcing $\delta$'s is $(4,+\infty)$ and thus $\mathcal D(\beta, M_2) = (1,4)$.

\item $M_3=\left( \begin{smallmatrix} 0& 0\\ 1&1 \end{smallmatrix} \right)$\\
By Lemma~\ref{le:LambdaKappa} we need to check only the values $k = 2 \kappa$, $\ell = 2\lambda$, where $\lambda +\kappa$ is even. Item 4 gives $ \kappa\geq \lambda$. Items 1 and 6  restrict   $\left( \begin{smallmatrix}  \ell \\ k \end{smallmatrix} \right)\in \{\left( \begin{smallmatrix}  0 \\ 4 \end{smallmatrix} \right), \left( \begin{smallmatrix}  2 \\ 2 \end{smallmatrix} \right) \}$ and $m < 4$. The inequality in  $\mathfrak{P}3$ is fulfilled only for $k=4, \ell=0$ and $1 \leq m<4$.  

Analogously  to the previous case,  $\mathcal D(\beta, M_3) = (1,4)$. 
 \end{itemize}
Note an interesting fact.  If we replace the value $\beta =\frac12$ by some $\beta_{-} < \frac12$,  the triplets $m=0, k= 2, \ell = 2$ and $m=2, k= 2, \ell = 2$ extend the set of $(1+\beta_{-})$-forcing $\delta$'s and $\mathcal D(\beta_{-}, M_3) =\emptyset$.

\end{example}

\section{Graphs of  admissible tails}
In this section we create the main tool that enables us to find balanced sequences with a small value of $E^*$. More precisely, for given $P=\Per\, \yy$ and $P'=\Per\,\yy'$ and $\beta >0$,  we will be able to identify candidates for suffixes (we call them ``admissible tails'')  of the continued fraction expansion of the slope $\theta$ of a~Sturmian sequence $\uu$ whose colouring $\vv=\colour(\uu,\yy,{\yy'})$ satisfies  $E^*(\vv)\leq 1+\beta$.

\begin{definition}\label{def:graph} Given $\beta >0$ and $P, \ P' \in \mathbb N$. An oriented graph $(V,E)$ is called the graph of $(1+\beta)$-admissible tails, and denoted  $\Gamma_\beta$, if 
\begin{itemize}
    \item the set of vertices  $V$ consists of classes of equivalence $\equiv$;
    \item a pair $\bigl([A]_\equiv, [B]_\equiv \bigr)$ labeled by $a$  belongs to  the set $E$ of oriented edges   if  $a = \lfloor \delta \rfloor $ for some  $\delta \in  \mathcal{D}(\beta, A) $ and  $ B \in  [ A \left( \begin{smallmatrix} 0 & 1\\  1 & a \end{smallmatrix} \right)]_\equiv$.
\end{itemize}

\end{definition}

 The graph $\Gamma_\beta$ is finite: it has a finite number of vertices as the number of classes of equivalence is  finite and a finite number of edges as by Lemma \ref{omez} the set ${\mathcal D}(\beta, A)$ is bounded for each class $[A]_\equiv$. Moreover, if $\beta_1>\beta_2$, then $\mathcal{D}(\beta_2, A) \subset \mathcal{D}(\beta_1, A)$ for every unimodular matrix $A$. Hence $\Gamma_{\beta_2}$ is a subgraph of $\Gamma_{\beta_1} $.

\medskip 
Let us recall that an oriented infinite path in a graph  $(V,E)$ is a sequence $v_0,e_0,v_1,e_1,v_2,e_2, \ldots$ such that $v_i \in V, e_i \in E$   and  $(v_i, v_{i+1}) = e_i$ for every $i \in \N$.    
  
The graph terminology allows us to rephrase Corollary \ref{cesta}  into the following theorem.

\begin{theorem}\label{Thm:hlavni}
Let $\beta>0$ and $\vv=\colour(\uu,\yy,\yy')$, where $\uu$  is a Sturmian sequence  with slope $\theta = [0, a_1,a_2, a_3, \ldots]$. 
Assume that   $E^*(\vv) \leq 1+\beta$.

 Then
 there exists  an infinite oriented path $v_0,e_0,v_1,e_1,v_2,e_2, \ldots$ in $\Gamma_\beta$ and  $N_0 \in \N$ such that 
 for every  $N \in \N$
\begin{enumerate} \item $[A_{N+N_0}]_\equiv =   v_{N}$;
\medskip  \item   $a_{N+N_0+1}$ is the label of the edge $e_{N}$.\end{enumerate}
\end{theorem}

In the graph  $\Gamma_\beta$ of admissible tails we are interested in infinite paths  on which  vertices occur infinitely many times. Therefore it is enough to consider strongly connected components, i.e., subgraphs with an oriented path from each vertex to each vertex.   In particular, if a component has only one vertex, it has to have a loop.  

\begin{corollary} Let $\beta >0$ and $P, \ P' \in \mathbb N$. If $\Gamma_\beta$ contains no oriented cycle, then $E^*(\vv) > 1+\beta$ for every colouring $\vv$ of a Sturmian sequence by constant gap sequences of periods $P$ and $P'$. 

\end{corollary}

\begin{example} \label{5letters} 
Let us construct the graph of $(1+\beta)$-admissible tails $\Gamma_\beta$ for parameters $P=2, P'=4$ and $\beta=\frac{1}{2}$. 
By  Example \ref{Ex: 24}, the graph has three vertices with  $(Y,Y')$-names 
$$M_1=\left( \begin{smallmatrix} 0& 0\\ 1&0 \end{smallmatrix} \right), M_2=\left( \begin{smallmatrix} 0& 0\\ 0&1 \end{smallmatrix} \right), M_3=\left( \begin{smallmatrix} 0& 0\\ 1&1 \end{smallmatrix} \right)$$
and the corresponding sets 
$$\mathcal D(\beta, M_1) = \emptyset, \quad \mathcal D(\beta, M_2) = (1,4), \quad \mathcal D(\beta, M_3) = (1,4).$$
The graph $\Gamma_\beta$ and its only strongly connected component are depicted in Figure~\ref{fig:graph_component242}.

By Theorem~\ref{Thm:hlavni} the suffix $\overline{2}$ is the only candidate for the suffix of $\theta$ associated with a~Sturmian sequence $\uu$ such that $\vv=\colour(\uu,\yy,\yy')$ with $\Per\,\yy=2, \ \Per\,\yy'=4$ has $E^*(\vv)\leq \frac{3}{2}$.
And indeed, $E^*(\vv)=\frac{3}{2}$ for $\theta=[0,1,\overline{2}]$. The reader is invited to check that $\Phi_N=\frac{1}{2}$ in Proposition~\ref{vzorec} (the maximum is attained for $m=0$, $k=\ell=2$).

As we have noticed in Example \ref{Ex: 24},  if we choose instead of $\beta = \frac12$ the value $\beta_{-} < \tfrac12$, then $D(\beta_{-}, M_3) =\emptyset$ and the corresponding graph contains no oriented cycle. In other words, we can see immediately that
there is no $\vv=\colour(\uu,\yy,\yy')$ with $E^*(\vv)< \frac{3}{2}$ for  $\Per\,\yy=2, \ \Per\,\yy'=4$.

 \begin{figure}[t] 
 \centering
 \hfill
 \begin{minipage}[t]{0.45\textwidth}
%   \subfigure[]{\includegraphics{graph242.pdf}} 
\subfigure[]{
    \begin{tikzpicture}[node distance={30mm}, thick, main/.style = {}] 
	    \node[main] (1) {$\mat{0}{0}{1}{0}$}; 
	    \node[main] (2) [below right of = 1] {$\mat{0}{0}{0}{1}$}; 
	    \node[main] (3) [above right of = 2] {$\mat{0}{0}{1}{1}$}; 
	    \draw[->] (3) to [out=170,in=10,looseness=1] node[midway, above, sloped] {1} (1);
    	\draw[->] (3) to [out=190,in=-10,looseness=1] node[midway, above, sloped] {3} (1);
	    \draw[->] (2) --node[midway, above, sloped] {2}  (1); 
	    \draw[->] (2) --node[midway, above, sloped] {1}  (3); 
	    \draw[->] (2) to [out=20,in=270,looseness=1] node[midway, above, sloped] {3} (3);
    	\draw[->] (3) to [out=90,in=10,looseness=3] node[midway, above, sloped]{2} (3);
\end{tikzpicture}
}
 \end{minipage}
 \hfill
 \begin{minipage}[t]{0.4\textwidth}
%   \subfigure[]{\includegraphics{component242v2.pdf}}
    \subfigure[]{
	    \begin{tikzpicture}[node distance={30mm}, thick, main/.style = {}] 
		    \node[main] (3) {$\mat{0}{0}{1}{1}$}; 
            \draw[->] (3) to [out=90,in=10,looseness=3] node[midway, above, sloped]{2} (3);
%       \draw[->] (3) to [out=130,in=50,looseness=3] node[midway, above, sloped]{2} (3);
    	\end{tikzpicture} 
	}
 \end{minipage}

%   \subfigure[]{\includegraphics{graph242.pdf}} 
%   \subfigure[]{\includegraphics{component242.pdf}}
 %   \resizebox{10 cm}{!}{\includegraphics{graph242.pdf}
    % \includegraphics{component242.pdf}}
    
\caption{(a) \ The graph of admissible tails for $\beta=\frac{1}{2}$ and $P=2, \ P'=4$ and (b) \ its unique strongly connected component.}
    \label{fig:graph_component242}
\
\end{figure}
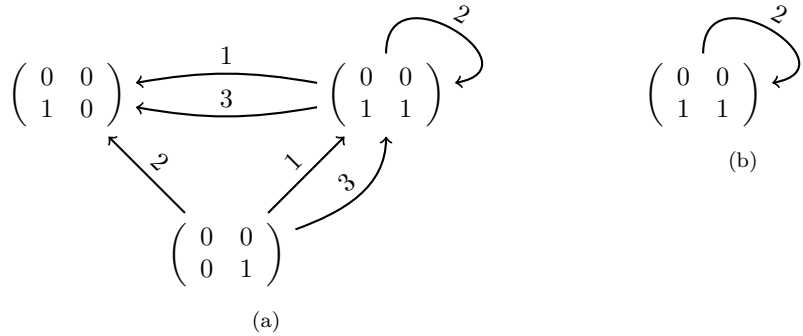

\end{example}

\begin{example} \label{ex:min4letters}\label{rem:Pelto}

The authors of \cite{RSV19} show that the least  critical exponent $E(\vv)$ on 4 letters -- in our notation $RTB(4)$ --  equals  $1+\frac{\sqrt{5}+1}{4}$ and it is reached for the colouring of the Fibonacci sequence by constant gap sequences with $P=P'=2$.
Let us deduce that $RTB^*(4) =  RTB(4)$.

Thanks to Proposition \ref{symmetry}, Examples~\ref{ex:4lettersperiod13} and  \ref{ex:4lettersperiod22},  and Remark~\ref{rem:periods_const_gap}, it remains to inspect only the case  $P=1$ and $P'=4$. 

There are six classes of equivalent matrices with $(Y,Y')$-names 
$$M_1=\left( \begin{smallmatrix} 0& 0\\ 1&0 \end{smallmatrix} \right), M_2=\left( \begin{smallmatrix} 0& 0\\ 0&1 \end{smallmatrix} \right), M_3=\left( \begin{smallmatrix} 0& 0\\ 1&1 \end{smallmatrix} \right), M_4=\left( \begin{smallmatrix} 0& 0\\ 2&1 \end{smallmatrix} \right), M_5=\left( \begin{smallmatrix} 0& 0\\ 1&2 \end{smallmatrix} \right), M_6=\left( \begin{smallmatrix} 0& 0\\ 1&3 \end{smallmatrix} \right).$$
Let us write down for each $M_i$ the triplets $m, k, \ell$ that influence the set $\mathcal{D}(\beta, M_i)$ (the other triplets satisfying Property $\mathfrak{P}1$ and $\mathfrak{P}3$ do not reduce $\mathcal{D}(\beta, M_i)$ any more): 
$$\begin{array}{rcll}
M_1=\left( \begin{smallmatrix} 0& 0\\ 1&0\end{smallmatrix} \right) & m=0, k=1,\ell=0 & \mathcal{D}(\beta, M_1)=\emptyset\\
M_2=\left( \begin{smallmatrix} 0& 0\\ 0&1 \end{smallmatrix} \right)& m=0, k=0,\ell=1 & \mathcal{D}(\beta, M_2)=\emptyset\\
M_3=\left( \begin{smallmatrix} 0& 0\\ 1&1 \end{smallmatrix} \right)& m=2, k=1,\ell=1 & \mathcal{D}(\beta, M_3)=(1,3)\\ 
M_4=\left( \begin{smallmatrix} 0& 0\\ 2&1 \end{smallmatrix} \right)& m=1, k=1,\ell=1 & \mathcal{D}(\beta, M_4)=(1,2)\\ 
M_5=\left( \begin{smallmatrix} 0& 0\\ 1&2 \end{smallmatrix} \right)& m=1, k=2,\ell=0 & \mathcal{D}(\beta, M_5)=(1,2)\\ 
M_6=\left( \begin{smallmatrix} 0& 0\\ 1&3 \end{smallmatrix} \right)& m=0, k=1,\ell=1 & \mathcal{D}(\beta, M_6)=\emptyset
\end{array}$$

 \begin{figure}[t]     
 \centering
%   \resizebox{8 cm}{!}{     \includegraphics{graph14.pdf}}
\begin{tikzpicture}[node distance={25mm}, thick, main/.style = {}]
	\node[main] (2) [] {$\mat{0}{0}{1}{1}$};
	\node[main] (4) [below of = 2] {$\mat{0}{0}{1}{2}$};
	\node[main] (3) [right of = 4] {$\mat{0}{0}{2}{1}$};
	\node[main] (5) [right of = 2] {$\mat{0}{0}{1}{3}$};
	\node[main] (0) [left of = 2] {$\mat{0}{0}{0}{1}$};	
	\node[main] (1) [below of = 0] {$\mat{0}{0}{1}{0}$};
	\draw[->] (2) --node[midway, above, sloped] {1} (4);
	\draw[->] (2) --node[midway, above, sloped] {2} (5);
	\draw[->] (3) --node[midway, above, sloped] {1} (5);
	\draw[->] (4) --node[midway, above, sloped] {1} (3);
\end{tikzpicture}
 \caption{The graph of admissible tails for $\beta=\frac{\sqrt{5}+1}{4}$ and $P=1, \ P'=4$.}
    \label{fig:graph14}
\end{figure}
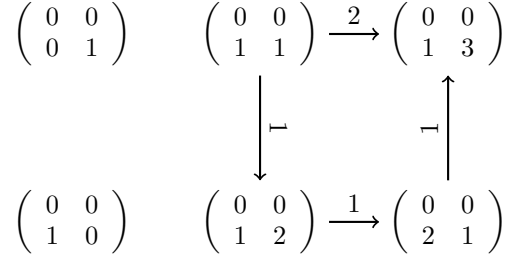
In the graph $\Gamma_\beta$ with parameters  $P=1, \ P'=4$ and $\beta=\frac{\sqrt{5}+1}{4}$ (see Figure~\ref{fig:graph14}) there is no strongly connected component. This completes the proof that $RTB^*(4)=1 + \frac{\sqrt{5}+1}{4}$.
\end{example}

\section{Asymptotic repetitive  threshold for binary to quinary balanced sequences}\label{2to5}
Using graphs of admissible tails we are able to list the least asymptotic critical exponent of $d$-ary balanced sequences for $2 \leq d \leq 5$.
\begin{itemize}
 
    \item It is known that $RTB^*(2)=RTB(2)=2 + \frac{\sqrt{5}+1}{2}\doteq3.618$ and it is reached for the Fibonacci sequence.
    \item $RTB^*(3)=2 + \frac{1}{\sqrt{2}}$: \ \ It suffices to consider $P=1, P'=2$ and $\beta=1+\frac{1}{\sqrt{2}}$. There are three classes of equivalent matrices with $(Y, Y')$-names: 
$$M_1=\left( \begin{smallmatrix} 0& 0\\ 1&0 \end{smallmatrix} \right), M_2=\left( \begin{smallmatrix} 0& 0\\ 0&1 \end{smallmatrix} \right), M_3=\left( \begin{smallmatrix} 0& 0\\ 1&1 \end{smallmatrix} \right).$$

Let us write down for each $M_i$ the triplets $m, k, \ell$ that influence the set $\mathcal{D}(\beta, M_i)$: 
$$\begin{array}{rcll}
M_1=\left( \begin{smallmatrix} 0& 0\\ 1&0\end{smallmatrix} \right) & m=1, k=1,\ell=0 & \mathcal{D}(\beta, M_1)=(1,2)\\
M_2=\left( \begin{smallmatrix} 0& 0\\ 0&1 \end{smallmatrix} \right)& m=0, k=0,\ell=1 & \mathcal{D}(\beta, M_2)=\emptyset\\
M_3=\left( \begin{smallmatrix} 0& 0\\ 1&1 \end{smallmatrix} \right)& m=3, k=2,\ell=0 & \mathcal{D}(\beta, M_3)=(1,4)
\end{array}$$
There is a unique strongly connected component in the graph $\Gamma_\beta$, which is the same as the one depicted in Figure~\ref{fig:graph_component242}. By Theorem~\ref{Thm:hlavni} the suffix $\overline{2}$ is the only candidate for the suffix of the continued fraction $\theta$ associated with a~Sturmian sequence $\uu$ such that $\vv=\colour(\uu,\yy,\yy')$ with $\Per\,\yy=1, \ \Per\,\yy'=2$ has $E^*(\vv)\leq 2+\frac{1}{\sqrt{2}}$.
And indeed, $E^*(\vv)=2+\frac{1}{\sqrt{2}}$ for $\theta=[0,1,\overline{2}]$. The reader is invited to check that $\Phi_N=\frac{2+\frac{Q_{N-1}}{Q_N}}{1+\frac{Q_{N-1}}{Q_N}}$ (the maximum in \eqref{Eq:Phi} is attained for $m=1, \ k=0, \ \ell=1$)  and $Q_{N+1}=2Q_N+Q_{N-1}$ for $N\geq 1$. Using Proposition~\ref{vzorec} we have $E^*(\vv)=1+\lim_{N\to \infty}\Phi_N=2+\frac{1}{\sqrt{2}}$.

   \item $RTB^*(4)=1 + \frac{\sqrt{5}+1}{4}$: \ \ It was  shown in Example~\ref{ex:min4letters}.
    \item $RTB^*(5)=\frac{3}{2}$: \ \ We set $\beta=\frac{1}{2}$. By Remark~\ref{rem:periods_const_gap}, Proposition \ref{symmetry}  and Lemma~\ref{delitele} we  have to inspect the pairs  
    $(P,P') \in \{(2,3), (2,4), (1,6), (1,8)\}$. In all cases besides $(2,4)$ there are no strongly connected components in the graph $\Gamma_\beta$. 
    The case of $P=2,\ P'=4$ was examined in Example~\ref{5letters}, where it was shown that the unique admissible tail is $\overline{2}$ and $E^*(\vv)=\frac{3}{2}$ for $\theta=[0,1,\overline{2}]$.
\end{itemize}

\medskip 
The method used  to determine  $RTB^*(d)$ for $d\leq 5$ does not work for  $d=6$. Let us explain why. 
To find $RTB^*(6)$ we have to inspect the pairs 
$(P,P') \in \{(3,3), (3,4),  (2,6), (2,8), (1,5),  (1,9), (1,12), (1,16)\}$. In the sequel, we will show that  $RTB^*(6)=1+ \frac{3\sqrt{65}-5}{80}\doteq 1.2398$. If we  construct  the graph $\Gamma_\beta$  with the optimal value   $\beta = \frac{3\sqrt{65}-5}{80} $,  
we find out that in  all cases besides $(P,P')=(1,16)$ there are no strongly connected components in  $\Gamma_\beta$. 

Hence we focus on the case $(P,P')=(1,16)$.  For this pair, 
    the strongly connected component of $\Gamma_\beta$ is depicted in Figure~\ref{fig:6pismen} (we will show later that the bold cycle corresponds to the unique $(1+\beta)$-admissible tail).
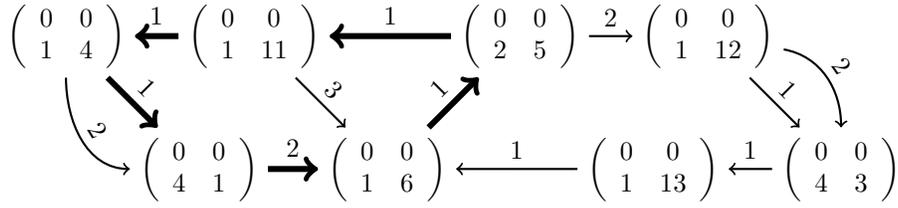
\begin{figure}    

 \centering
%   \resizebox{12 cm}{!}{     \includegraphics{graph6_pismen_1_16_cervene.pdf}}
    \input{graph1_16_before}
 \caption{The strongly connected component of the graph of admissible tails for $\beta=\frac{3\sqrt{65}-5}{80}$ and $P=1, \ P'=16$.}
    \label{fig:6pismen}
\end{figure}

Although $\overline{1,1,3}$ is the label of an infinite path in $\Gamma_\beta$, the asymptotic critical exponent $E^*(\vv) \geq 1.41>1+\beta$ for any colouring of a Sturmian sequence with slope $\theta$ having eventually periodic continued fraction with the period $\overline{1,1,3}$.  In other words, 
the implication in Theorem~\ref{Thm:hlavni} cannot be reversed. 
Our next goal is to reduce the graph $\Gamma_\beta$ in order to exclude  paths corresponding to the asymptotic critical exponent  $>1+\beta$.

\section{Reduction of the graphs}

We have seen that the existence of an infinite path in the graph of admissible tails does not guarantee a small value of $E^*$. There may be even uncountably many paths in the graph $\Gamma_\beta$ corresponding to $E^*(\vv) > 1+\beta$.  However, we will show how to reduce the graph so that the number of unsuitable paths is diminished. This method enables us to find $RTB^*(d)$ for $6 \leq d \leq 10$.

We describe two reasons why the implication in Theorem~ \ref{Thm:hlavni} cannot be reversed.
\begin{enumerate}
    \item

Let  $d_1,d_2,d_3,\ldots$ be labels of  edges in  an infinite path  $v_0,e_0,v_1,e_1,v_2,e_2, \ldots$ in $\Gamma_\beta$ and all vertices of the path occurs in it infinitely many times.  If 
    the path corresponds to an asymptotic critical exponent $\leq 1+\beta$, then for each $N \in \mathbb{N}$, the continued fraction $[d_{N+1}, d_{N+2}, \ldots] $  belongs to $\mathcal{D}(\beta, A)$, where $A \in v_N$.  But the  construction of the graph requires only $d_{N+1} =\lfloor \delta\rfloor$ for some $\delta \in \mathcal{D}(\beta, A)$.  

Hence,  if  it is not possible to extend the edge starting in $v_{N}$ and labeled by $d_{N+1}$ to an infinite path whose labels make the continued fraction expansion of some $\delta \in \mathcal{D}(\beta, A)$, then this edge may be erased without influencing the validity of Theorem~\ref{Thm:hlavni}. 

We call this process {\em forward reduction}.

\item  The terms used to  describe Property $\mathfrak{P}3$ in Definition~\ref{Forcing} represent the  minimal value  of the function $f_{m,k,\ell}(x) = \frac{1+m + x}{k+\ell m+ \ell x}$ for  $x \in [0,1]$, see \eqref{funkceF}. The actual value of $x$  we use in evaluation of $\Phi_N$ in \eqref{FI} is $x_N=Q_{N-1}/Q_N =[0,a_N,a_{N-1}, a_{N-2}, \ldots, a_1]$. 
It may happen that $f_{m,k,\ell}(x_N)> \beta$ even if the minimum of  $f_{m,k,\ell}$ on $[0,1]$ is smaller than $\beta$. 

To construct $\Gamma_\beta$ we use only the fact that $x_N  \in [0,1]$. 
As  $x_N=[0,a_N,a_{N-1}, a_{N-2}, \ldots, a_1]$, the value $x_N$ is given by the labels of edges of a path of length $N$ ending in $v_N=[A]_\equiv$.  Assume that the structure of  $\Gamma_\beta$ enables to deduce that all paths entering the vertex $[A]_\equiv$ have  $x_N \in  J \subset [0,1] $.  If  for some $m,k,\ell$ satisfying  Property $\mathfrak{P}1$  the following inequality holds true
 $$  \min\{f_{m,k,\ell}(x): x \in J\} \ > \  \beta\,, \ \ \
$$
then $\delta$ satisfying Property $\mathfrak{P}2$ is $(1+\beta)$-forcing, i.e., all such $\delta$'s should be deleted from the set  $\mathcal{D}(\beta, A)$. The new set $\mathcal{D}'(\beta, A)$ may cause that some edges starting in the vertex $[A]_\equiv$ are deleted as well.

We call this process {\em backward reduction}.

\end{enumerate}
 We always apply the forward and backward reductions together with searching for strongly connected components as long as the graph changes.

Let us  illustrate both kinds of reductions on the following example.
\begin{example}\label{ex:343}
Let us construct $\Gamma_\beta$ for $P=3, \ P'=4,\ \beta = \tfrac13$. Then $H=1$ and $L=12$. There are 24 classes of equivalence with the following $(Y, Y')$-names:
$$\left( \begin{smallmatrix} 0& 1\\ 0&1 \end{smallmatrix} \right), \left( \begin{smallmatrix} 1& 0\\ 1&0 \end{smallmatrix} \right), \left( \begin{smallmatrix} 1& 0\\ 1&2\end{smallmatrix} \right),\left( \begin{smallmatrix} 1& 1\\ 2&1 \end{smallmatrix} \right), \left( \begin{smallmatrix} 1& 2\\ 1&3 \end{smallmatrix} \right), 
\left( \begin{smallmatrix} 0& 1\\ 1&1 \end{smallmatrix} \right), \left( \begin{smallmatrix} 0& 1\\ 1&0 \end{smallmatrix} \right), \left( \begin{smallmatrix} 1& 1\\ 1&0 \end{smallmatrix} \right), \left( \begin{smallmatrix} 1& 2\\ 1&0 \end{smallmatrix} \right), \left( \begin{smallmatrix} 0& 1\\ 1&3 \end{smallmatrix} \right), 
\left( \begin{smallmatrix} 1& 0\\ 1&3 \end{smallmatrix} \right), \left( \begin{smallmatrix} 1& 1\\ 1&3\end{smallmatrix} \right),$$
$$\left( \begin{smallmatrix} 1& 0\\ 0&1\end{smallmatrix} \right), \left( \begin{smallmatrix} 1& 0\\ 2&1\end{smallmatrix} \right), \left( \begin{smallmatrix} 1& 0\\ 1&1\end{smallmatrix} \right), \left( \begin{smallmatrix} 1& 2\\ 0&1\end{smallmatrix} \right), \left( \begin{smallmatrix} 1& 1\\ 0&1\end{smallmatrix} \right), \left( \begin{smallmatrix} 0& 1\\ 1&2\end{smallmatrix} \right), \left( \begin{smallmatrix} 1& 1\\ 1&2\end{smallmatrix} \right),\left( \begin{smallmatrix} 1& 2\\ 1&2\end{smallmatrix} \right),\left( \begin{smallmatrix} 1& 2\\ 1&1\end{smallmatrix} \right),\left( \begin{smallmatrix} 0& 1\\ 2&1\end{smallmatrix} \right),\left( \begin{smallmatrix} 1& 2\\ 2&1\end{smallmatrix} \right), \left( \begin{smallmatrix} 1& 1\\ 1&1\end{smallmatrix} \right).
$$
By Lemma \ref{omez},  $\mathcal{D}(\beta, M_i)\subset (1, 14)$. We write down for each class the triplets $m, k, \ell$ that influence the set $\mathcal{D}(\beta, M_i)$: 
%%%%%%%%%%%%%%%%%%Pokus o jiny format%%%%%%%%%%

$$\begin{array}{lll} & &  \textbf{list of triplets}\ (m,k,\ell)\\
&&\\
M_1=\left( \begin{smallmatrix} 0& 1\\ 0&1 \end{smallmatrix} \right)\qquad & \mathcal{D}(\beta, M_1)=\emptyset & \quad \ \ (0, 0,1)\\ 
M_2=\left( \begin{smallmatrix} 1& 0\\ 1&0 \end{smallmatrix} \right)& \mathcal{D}(\beta, M_2)=\emptyset& \quad \ \ (0,1,0) \\
M_3=\left( \begin{smallmatrix} 1& 0\\ 1&2\end{smallmatrix} \right) & \mathcal{D}(\beta, M_3)=\emptyset & \quad \ \ (0,2,0)\\
M_4=\left( \begin{smallmatrix} 1& 1\\ 2&1 \end{smallmatrix} \right) & \mathcal{D}(\beta, M_4)=\emptyset& \quad \ \ (0, 2,1)\\
M_5=\left( \begin{smallmatrix} 1& 2\\ 1&3 \end{smallmatrix} \right)  & \mathcal{D}(\beta, M_5)=\emptyset&\quad \ \  (0,1,1)\\
M_6=\left( \begin{smallmatrix} 0& 1\\ 1&1 \end{smallmatrix} \right)  & \mathcal{D}(\beta, M_6)=\emptyset& \quad \ \ (0,3,1)\\
M_7=\left( \begin{smallmatrix} 0& 1\\ 1&0 \end{smallmatrix} \right) & \mathcal{D}(\beta, M_7)=(1,2)& \quad \ \ (0,3,0)\\ 
M_8=\left( \begin{smallmatrix} 1& 1\\ 1&0 \end{smallmatrix} \right)& \mathcal{D}(\beta, M_8)=(1,2)&\quad \ \  (0,3,0 )\\
M_9=\left( \begin{smallmatrix} 1& 2\\ 1&0 \end{smallmatrix} \right) & \mathcal{D}(\beta, M_9)=(1,2)&\quad \ \  (0, 3,0 )\\
M_{10}=\left( \begin{smallmatrix} 0& 1\\ 1&3 \end{smallmatrix} \right) &\mathcal{D}(\beta, M_{10})=(1,2)  & \quad \ \ (1, 4, 2), (2,7,1) \\
 M_{11}=\left( \begin{smallmatrix} 1& 0\\ 1&3 \end{smallmatrix} \right)&\mathcal{D}(\beta, M_{11})=(1,4)&\quad \ \  (1,4,0)\\
M_{12}=\left( \begin{smallmatrix} 1& 1\\ 1&3\end{smallmatrix} \right) & \mathcal{D}(\beta, M_{12})=(1,\frac{5}{2})& \quad \ \ (1,4,1)\\
M_{13}=\left( \begin{smallmatrix} 1& 0\\ 0&1\end{smallmatrix} \right) &   \mathcal{D}(\beta, M_{13})=(1,3) \cup (\frac{7}{2}, 4)&\quad \ \ (1,4,0), (2,2,3)\\
M_{14}=\left( \begin{smallmatrix} 1& 0\\ 2&1\end{smallmatrix} \right) & \mathcal{D}(\beta, M_{14})=(\frac{3}{2}, 4)&\quad \ \ (0,2,3), (1,4,0)\\
M_{15}=\left( \begin{smallmatrix} 1& 0\\ 1&1\end{smallmatrix} \right)& \mathcal{D}(\beta, M_{15})=(1,2) \cup (\frac{5}{2}, 4)&\quad \ \ (1,2,3), (1,4,0)\\
M_{16}=\left( \begin{smallmatrix} 1& 2\\ 0&1\end{smallmatrix} \right)& \mathcal{D}(\beta, M_{16})=(1,2)&\quad \ \ (1,3,1)\\
M_{17}=\left( \begin{smallmatrix} 1& 1\\ 0&1\end{smallmatrix} \right)  & \mathcal{D}(\beta, M_{17})=(1,2)\cup (4,\frac{9}{2})&\quad \ \ (1,2,2),(2,6,1)\\
M_{18}=\left( \begin{smallmatrix} 0& 1\\ 1&2\end{smallmatrix} \right)& \mathcal{D}(\beta, M_{18})=\emptyset&\quad \ \ (0,3,2), (1,6,0), (2,5,2) \\
M_{19}=\left( \begin{smallmatrix} 1& 1\\ 1&2\end{smallmatrix} \right) & \mathcal{D}(\beta, M_{19})=(2,3)&\quad \ \ (0,1,2), (1,6,0), (2,3,2)\\
M_{20}=\left( \begin{smallmatrix} 1& 2\\ 1&2\end{smallmatrix} \right)&\mathcal{D}(\beta, M_{20})=(1,2)&\quad \ \ (1,3,2), (1,6,0), (3,5,2)\\
M_{21}=\left( \begin{smallmatrix} 1& 2\\ 1&1\end{smallmatrix} \right)&\mathcal{D}(\beta, M_{21})=(3,4)&\quad \ \ (0,2,2), (2,5,1)\\
M_{22}=\left( \begin{smallmatrix} 0& 1\\ 2&1\end{smallmatrix} \right)&\mathcal{D}(\beta, M_{22})=(1,3)&\quad \ \ (1,5,1)\\
M_{23}=\left( \begin{smallmatrix} 1& 2\\ 2&1\end{smallmatrix} \right)&\mathcal{D}(\beta, M_{23})=(1,3)&\quad \ \ (2,4,2), (3,7,1)\\
M_{24}=\left( \begin{smallmatrix} 1& 1\\ 1&1\end{smallmatrix} \right)& \mathcal{D}(\beta, M_{24})=(1,4)&\quad \ \ (3,4,2), (3,8,1)
\end{array}$$

\medskip

A unique strongly connected component of the graph $\Gamma_\beta$ is depicted in Figure~\ref{fig:343equality}.
Using our computer program we can see that for the Sturmian sequence $\uu$ associated to $\theta=[0,3,\overline{1,1,1,2}]$, the balanced sequence $\vv=\colour(\uu,\yy,\yy')$ with $\Per\ \yy=3, \ \Per \ \yy'=4$ has $E^*(\vv)=\frac{4}{3}$.

% The implication in Theorem~\ref{Thm:hlavni} cannot be reversed. Although $(1,1,1,3,1,1,3,1,3)^{\omega}$ is the label of an infinite path in $\Gamma_\beta$, using our computer program, we can see that for each $\theta$ with the suffix $(1,1,1,3,1,1,3,1,3)^{\omega}$ the colouring $\vv=\barva{(\uu,\yy,\yy')}$ has 
% % $E^*(\vv) \geq \frac{940 + \sqrt{162544}}{928} \doteq 1,44738$. 
% $E^*(\vv) \geq \frac{407 + \sqrt{63505}}{456} \doteq 1,44518$.
% In the sequel, we will explain how to reduce the graph $\Gamma_\beta$ in order to see that $(1,1,1,2)^{\omega}$ is the only $(1+\beta)$-admissible suffix of $\theta$. 

If we search for $\vv=\colour(\uu,\yy,\yy')$ such that $E^*(\vv)< \frac{4}{3}$, i.e., if we choose $\beta_{-} < \tfrac13$, then by Theorem~\ref{BetaForcing} we have to exclude also the solution $m=0$, $k=\ell=3$, which reduces $\mathcal D(\beta_{-}, M_{10})$ to $\emptyset$, therefore no strongly connected component remains in the graph. To summarize, there is no $\vv=\colour(\uu,\yy,\yy')$ with $E^*(\vv)< \frac{4}{3}$ for $\Per\,\yy=3, \ \Per\,\yy'=4$.

\begin{figure}
    \centering
  
% \resizebox{8 cm}{!}{
    % \includegraphics{component343.pdf}}
    \input{graph3_4_3_before}
    \caption{The strongly connected component of the graph of admissible tails for $\beta=\frac{1}{3}$ and $P=3, \ P'=4$.}
    \label{fig:343equality}
\end{figure}
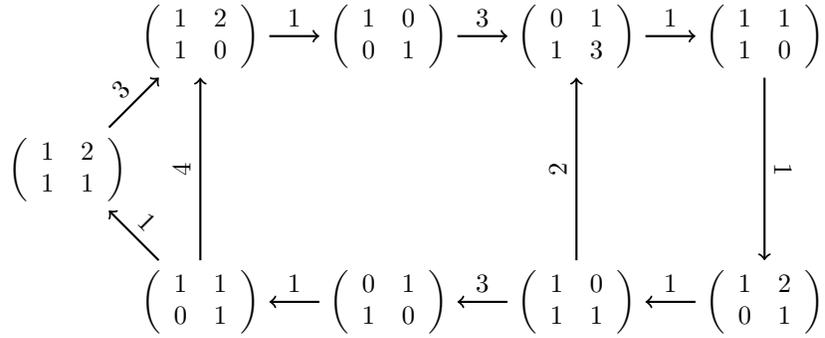

Let us apply first the forward reduction on the graph from Figure~\ref{fig:343equality}.
Consider the vertex $M_{17}=\left( \begin{smallmatrix} 1& 1\\ 0&1\end{smallmatrix} \right)$ and the outgoing edge labeled by 4. Each prolongation to an infinite path has the next edge label equal to 1. Therefore the corresponding $\delta=[4,1,\ldots] \in (\frac{9}{2},5)$, i.e., 
$\delta \not\in \mathcal{D}(\beta, M_{17})=(1,2)\cup (4,\frac{9}{2})$. Consequently, the edge labeled by 4 may be erased.

Next, we apply the backward reduction.
Consider the vertex $M_{21}=\left( \begin{smallmatrix} 1& 2\\ 1&1\end{smallmatrix} \right)$. The sequence of edge labels of each ingoing path ends in $3,1,1$. For $m=1, k=6, \ell=1$ satisfying $M_{21}\left( \begin{smallmatrix} 1& 0\\ m&1\end{smallmatrix} \right)\left( \begin{smallmatrix} \ell\\ k\end{smallmatrix} \right)=\left( \begin{smallmatrix} 0\\0\end{smallmatrix} \right)\mod \left( \begin{smallmatrix} 3\\4\end{smallmatrix} \right)$ we have $[0,1,1,3,\ldots] \in J=[\frac{5}{9}, \frac{4}{7}]$ and $ \min\{f_{m,k,\ell}(x): x \in J\}=\frac{2+5/9}{7+5/9} = \frac{23}{68}>\beta=\frac{1}{3}$. Therefore, by Remark~\ref{obs:simplification} the triplet $m=1, k=6, \ell=1$ reduces $\mathcal{D}(\beta, M_{21})$ to  $\mathcal{D}'(\beta, M_{21}) =(3,4) \cap (1,\frac{7}{2})= (3,\frac{7}{2})$.

Finally, we apply once more the forward reduction. Consider again the vertex $M_{21}$, now with $\mathcal{D}'(\beta, M_{21})=(3,\frac{7}{2})$, and its unique outgoing edge labeled by 3. Each prolongation to an infinite path has the next edge label equal to 1. Therefore the corresponding $\delta=[3,1,\ldots] \in (\frac{7}{2},4)$, i.e., 
$\delta \not\in \mathcal{D}'(\beta, M_{21})=(3,\frac{7}{2})$. Thus, the edge labeled by 3 may be erased.
Consequently, the graph is reduced so that it contains a unique cycle labeled by $\overline{1,1,1,2}$, which proves that it is the only $(1+\frac{1}{3})$-admissible suffix of $\theta$.  
\end{example}

\section{Asymptotic repetitive threshold for senary to denary balanced sequences}
In order to find $RTB^*(d)$,  we first detect  possible periods $P=\Per\,\yy, P'= \Per\,\yy'$ which give a $d$-ary balanced sequence.  For this task we  use Remark~\ref{rem:periods_const_gap}.  Then some pairs are eliminated due to    Proposition~\ref{symmetry}  and 
Lemma~\ref{delitele}. 

The starting parameter $\beta$ for the construction of graphs of $(1+\beta)$-admissible tails can be  chosen to be the value of the minimal critical exponent since by definition $E(\vv)\geq E^*(\vv)$ for every  sequence $\vv$.

By Proposition~\ref{prop:bounds} we have $E^*(\vv) \geq 1+\frac{1}{P\,P' }$ for every balanced sequence $\vv=\colour(\uu,\yy,\yy')$. Therefore when searching for $E^*(\vv)\leq 1+\beta$ we have to consider only the pairs $(P, P')$ such that $\frac{1}{\beta} \leq {P \,P'}$.

As soon as we find $\vv$ with $E^*(\vv)< 1+\beta$, we lower $\beta$ to $\beta'=E^*(\vv)-1$ and construct graphs of $(1+\beta')$-admissible tails for the remaining pairs $(P,P')$.

\medskip
Using the above described procedure, we have found the minimal asymptotic critical exponent of $d$-ary balanced sequences up to $d=10$. They are listed in  Table~\ref{RTB*}.

\begin{table}[t]
\centering
\caption{The asymptotic repetitive threshold  $RTB^*(d)$ for alphabets of size  $d\leq 10$ and  the parameters of a~$d$-ary balanced sequence $\vv=\colour(\uu, \yy, \yy')$ for which the threshold $RTB^*(d)$ is reached:   $\theta$ is the slope of a Sturmian sequence $\uu$, $P=\Per\,\yy$ and $P'=\Per\,\yy'$.}\label{RTB*}
\setlength{\tabcolsep}{3.5pt}
\renewcommand{\arraystretch}{1.5}
\resizebox{0.7\textwidth}{!} {
\begin{tabular}{|r|l|l|r|r|}
\hline
$d$ &  $RTB^*(d)$&$ \theta$ & $P$ & $P'$  \\
\hline
\hline
2  &$2 + \frac{\sqrt{5}+1}{2}\doteq3.618034$ & $[0, \overline{1}]$ & 1 & 1 \\
\hline

3  & $2 + \frac{1}{\sqrt{2}}\doteq 2.707107$ & $[0, 1,\overline{2}]$ & 1 & 2  \\
\hline
4 & $1 + \frac{\sqrt{5}+1}{4}\doteq 1.809017$ & $[0, \overline{1}]$ & 2 & 2 \\
\hline
5& $\frac{3}{2}=1.5$ & $[0, 1,\overline{2}]$ & 2 & 4  \\
\hline
6& $\frac{75+3\sqrt{65}}{80}\doteq 1.239835 $  &
$[0,4,\overline{1,2,1,1,1}]$
% $[0, 1,3,\overline{1,2,1,1,1}]$
& 1 & 16  \\
\hline
7& $\frac{49 + \sqrt{577}}{64}\doteq 1.140950$ & $[0, 5,1, \overline{1,1,1,5,2}]$ & 1 & 32   \\
\hline
8& $1+\frac{3-\sqrt{5}}{16}\doteq 1.047746 $ & $[0, \overline{1}]$ & 8 & 8  \\
\hline
9  & $\frac{21-\sqrt{20}}{16}\doteq 1.032992$  & $[0, 1,\overline{4}]$ & 8 & 16 \\
\hline
10 & $\frac{364-21\sqrt{7}}{304}\doteq 1.0146027$  & 
$[0,6,\overline{1,1,1,1,2,1,2,1,1,1}]$
% $[0,1,3, \overline{1,2,1,2,1,1,1,1,1,1}]$ 
& 4 & 64 \\
\hline
\end{tabular}
}

\end{table}

Let us comment properties of the graphs $\Gamma_\beta$ with  the optimal value $\beta=RTB^*(d)-1$. 
\begin{description}
\item[$2 \leq d \leq 5$\ ] 
It was sufficient to use the graphs of admissible tails without reduction, as explained in Section~\ref{2to5}. 
\item[$6\leq d\leq 9$\ ] The  reduction was necessary. After reduction of the graph $\Gamma_\beta$ constructed  for the pair $(P, P')$ from Table~\ref{RTB*},  there is always a unique strongly connected component in the form of a cycle, i.e., the period of the continued fraction was determined uniquely by the graph.  For all other pairs $(P,P')$ of admissible periods the graph has no strongly connected component. Let us point out that for 6 letters the forward reduction suffices, while for more letters both the forward and the backward reduction is needed. For 6 letters, the unique $(1+\beta)$-admissible tail corresponds to the bold cycle in Figure~\ref{fig:6pismen}.
\item[$d=10$\ ]
 For $(P,P') \neq (4, 64)$ the graph  has no strongly connected component. For  $P=4,P'=64$,  a unique strongly connected component of $\Gamma_\beta$ after reduction is depicted in Figure~\ref{fig:graf10pismen}. 
 Thus, there are two oriented cycles sharing two vertices.  Hence a ``human intervention''  is needed to pick up the suitable continued  fraction. Using our computer program we obtain: 
$E^*(\hat\vv)=RTB^*(10)$ for the balanced sequence $\hat\vv=\colour(\hat\uu, \yy, \yy')$, where
$\hat\theta=[0,6,\overline{1,1,1,1,2,1,2,1,1,1}]$. 
 \begin{figure}[b]
    \centering
% \resizebox{12 cm}{!}{
%     \includegraphics{graph4_64_68.pdf}}
    \input{graph4_64_after}
    \caption{The strongly connected component of the graph of admissible tails for $\beta=\frac{60-21\sqrt{7}}{304}$ and $P=4, \ P'=64$.}
    \label{fig:graf10pismen}
\end{figure}
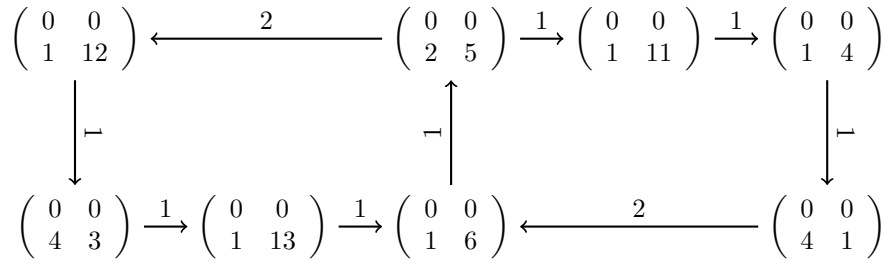
In Figure~\ref{fig:graf10pismen}, $\hat\theta$ corresponds to an infinite path going alternatively through the right hand cycle, then the left hand cycle, and again the right hand cycle, then the left hand cycle, and so on. 
 The following two arguments show that if $\theta$ does not have a suffix corresponding to such alternation of the right and left hand cycle, then $E^*(\vv)>1+\beta$. First we observe, that every  path  in the graph from   Figure \ref{fig:graf10pismen} uses infinitely many times the vertices named  $\left(\begin{smallmatrix}0& 0 \\ 1 & 6\end{smallmatrix}\right)$ and   $\left(\begin{smallmatrix}0& 0 \\ 2 & 5\end{smallmatrix}\right)$. 
   \begin{itemize}
    \item
Let us explain that any $\theta$ corresponding to an infinite path going infinitely many times twice consecutively through the left hand cycle gives rise to a colouring $\vv$ with $E^*(\vv) > 1+\beta$. 

Since $P=4, P'=64$, we have $H=4$ and $Y=1$, $Y'=16$.
Assume $\theta=[0,a_1,a_2,a_3,\ldots]$ and there exist infinitely many $N$
 such that $A_N$ belongs to the class with $(Y, Y')$-name $\left(\begin{smallmatrix}0& 0 \\ 1 & 6\end{smallmatrix}\right)$, where we arrive to $A_N$ using the left hand cycle and we leave $A_N$ using again the left hand cycle. In this case $\delta_N=[a_{N+1}, a_{N+2}, \ldots]=[1,2,1,1,1,1,\ldots] \in (\frac{29}{21}, \frac{18}{13})$ and $\frac{Q_{N-1}}{Q_N}=[0,a_N,\ldots, a_2,a_1]= [0,1,1,1,2,1,\ldots]\in (\frac{12}{19},\frac{7}{11})$. 
 Then $\left(\begin{smallmatrix}0& 0 \\ 1 & 6\end{smallmatrix}\right)\left(\begin{smallmatrix}14\\ 19\end{smallmatrix}\right)=\left(\begin{smallmatrix}0 \\ 0\end{smallmatrix}\right) \mod \left(\begin{smallmatrix}1 \\ 16\end{smallmatrix}\right)$, therefore $k=4\cdot 19=76, \ell=4\cdot 14=56$ belongs to ${\mathcal S}_1(N,0)$. 
 It belongs to ${\mathcal S}_2(N,0)$ as  the inequality $|56\delta_N-76|< \delta_N+1$ is satisfied as well. 
Using~\eqref{Eq:Phi} we obtain $$\Phi_N\geq\frac{1+\frac{Q_{N-1}}{Q_N}}{76+56\frac{Q_{N-1}}{Q_N}}\geq \frac{1+\frac{12}{19}}{76+56\frac{12}{19}}>1+\beta.$$ By Proposition~\ref{vzorec} $E^*(\vv)=1+\limsup \Phi_N >1+\beta$.

\item 
Let us  show that also any $\theta$ corresponding to an infinite path going infinitely many times twice consecutively through the right hand cycle gives rise to a colouring $\vv$ with $E^*(\vv) > 1+\beta$. 

Assume there exist infinitely many $N$
 such that $A_N$ belongs to the class with $(Y, Y')$-name $\left(\begin{smallmatrix}0& 0 \\ 2 & 5\end{smallmatrix}\right)$, where we leave $A_N$ using twice consecutively the right hand cycle. Hence   $\delta_N=[a_{N+1}, a_{N+2}, \ldots]=[1,1,1,2,1,1,1,1,2,1,\ldots] \in (\frac{305}{193}, \frac{177}{112})$ and  $\frac{Q_{N-1}}{Q_N}=[0,a_N,\ldots, a_2,a_1]= [0,1,a,\ldots]$, where $a=1$ or $a=2$, i.e., $\frac{Q_{N-1}}{Q_N} \in (\frac{1}{2},\frac{3}{4})$. 
The solution $k=72, \ell=44$ belongs to ${\mathcal S}(N,0)={\mathcal S}_1(N,0) \cap {\mathcal S}_2(N,0)$.
Using~\eqref{Eq:Phi} we obtain $$\Phi_N\geq \frac{1+\frac{Q_{N-1}}{Q_N}}{72+44\frac{Q_{N-1}}{Q_N}}\geq \frac{1+\frac{1}{2}}{72+44\frac{1}{2}}>1+\beta.$$ By Proposition~\ref{vzorec} $E^*(\vv)=1+\limsup \Phi_N >1+\beta$.

\end{itemize}

\noindent To sum up, only $\theta$ having the continued fraction with  the same period 
as $\hat\theta=[0,6,\overline{1,1,1,1,2,1,2,1,1,1}]$ can give the minimal value of $E^*(\vv)$. 
\end{description}

\section{Comments and open questions}
Let us first make some comments on complexity of computation of the asymptotic repetitive threshold. We were not able to determine $RTB^*(d)$ for $d\geq 11$ with our computer because of time complexity of the computation. The bigger $d$, the larger the number of pairs $(P,P')$ that have to be treated. The number of such pairs grows exponentially. Moreover, even the number of vertices in the graph grows exponentially with $d$. For example, for every $d$ we have to consider $P=1$ and $P'=2^{d-2}$. The number of vertices in the corresponding graph $\Gamma_\beta$ = the number of classes of equivalence $\equiv$ is equal to  $3\cdot 2^{d-3}$.

The closer the value $1+\beta$ to $RTB^*(d)$ the more time consuming determination of $\mathcal{D}(\beta, A)$. The reason is the fact that more triplets $m, k, \ell$ satisfy Properties $\mathfrak{P}1$ and $\mathfrak{P}3$. Hence, it would be useful to have a good upper bound on $RTB^*(d)$ so that we do not have to repeat the computation for more values $\beta$. 

Time complexity is not the only obstacle to revealing $RTB^*(d)$ for $d \geq 11$. We are afraid that for larger $d$ our method will not give a unique period of the continued fraction corresponding to the minimal $E^*(\vv)$ and that ``human intervention'' will be needed similarly as in the case $d=10$.  

Let us conclude with some other open problems connected to the topic. 
 
 \begin{itemize}
 \item  We conjecture that  $RTB^*(d) < 1+q^d$  for some positive $q<1$. But from the obtained values $RTB^*(d)$,   we are not able  to derive a conjecture on the precise value  $RTB^*(d)$.  We observe at least that if $(P,P')$ is an optimal pair of periods for $d$ even, then $(P,2P')$ is optimal for $d+1$. It works for $d=2,4,6,8$.
 
 \item We believe that  $RTB^*(d)$ always belongs to a quadratic field, but we have no proof of this fact.

\item  $RTB^*(d)$ is defined to be $\inf\{E^*(\uu):\uu \text{ \ $d$-ary balanced}\}$. We can see for $d \leq 10$ that $RTB^*(d)$ is minimum of the set and it is not its accumulation point. The question is whether $RTB^*(d)$ may be an accumulation point. And if not, what is the second, third, etc. smallest element of the set.

\end{itemize}

% \section{Acknowledgement} 
% The second author was supported  by Czech technical university in Prague, through the project  SGS20/183/OHK4/3T/14.
% The first and the third authors were supported by the Ministry of Education, Youth and Sports of the Czech Republic through the project  CZ.02.1.01/0.0/0.0/16\_019/0000778. 

% \bibliographystyle{elsarticle-harv.bst}

%% ----------------------- Pokud se citace zmeni, nutno znovu vytvorit .bbl file, tkere se pak sem prekopiruje

% \bibliographystyle{amsplain}
% \bibliography{biblio.bib}

%% ------------------------------------------------------


\end{document}

%% file: graph1_16_before.tex
\begin{tikzpicture}[scale = 0.1, node distance={25mm}, thick, main/.style = {}]	
	\node[main] (6)[] {$\mat{0}{0}{1}{6}$};
	\node[main] (15)[above right of = 6] {$\mat{0}{0}{2}{5}$};
	\node[main] (17)[right of = 15] {$\mat{0}{0}{1}{12}$};
	\node[main] (11)[below right of = 17] {$\mat{0}{0}{4}{3}$};
	\node[main] (13)[left of = 11] {$\mat{0}{0}{1}{13}$};
	
	\node[main] (5)[above left of = 6] {$\mat{0}{0}{1}{11}$};
	\node[main] (12)[left of = 5] {$\mat{0}{0}{1}{4}$};
	\node[main] (8)[below right of = 12] {$\mat{0}{0}{4}{1}$};
	
	\draw[->, line width=0.8mm] (5) --node[midway, above, sloped] {1} (12);
	\draw[->] (5) --node[midway, above, sloped] {3} (6);
	\draw[->, line width=0.8mm] (6) --node[midway, above, sloped] {1} (15);
	\draw[->, line width=0.8mm] (8) --node[midway, above, sloped] {2} (6);
	\draw[->] (11) --node[midway, above, sloped] {1} (13);
	\draw[->, line width=0.8mm] (12) --node[midway, above, sloped] {1} (8);
	\draw[->] (12) to [out=-90,in=180,looseness=1]  node[midway, above, sloped] {2} (8);
	\draw[->] (13) --node[midway, above, sloped] {1} (6);
	\draw[->, line width=0.8mm] (15) --node[midway, above, sloped] {1} (5);
	\draw[->] (15) --node[midway, above, sloped] {2} (17);
	\draw[->] (17) --node[midway, above, sloped] {1} (11);
	\draw[->] (17) to [out=-10,in=90,looseness=1]  node[midway, above, sloped] {2} (11);
\end{tikzpicture}

%% file: graph3_4_3_before.tex
% \begin{tikzpicture}[node distance={22mm}, thick, main/.style = {}]
% 	\node[main] (11)[] {$\mat{1}{2}{1}{1}$};
% 	\node[main] (6)[below left of = 11] {$\mat{1}{1}{0}{1}$};
% 	\node[main] (7)[below right of = 11] {$\mat{1}{2}{1}{0}$};
% 	\node[main] (1)[below left of = 6] {$\mat{0}{1}{1}{0}$};
% 	\node[main] (2)[below right of = 1] {$\mat{1}{0}{1}{1}$};
% 	\node[main] (12)[below of = 2] {$\mat{1}{2}{0}{1}$};
% 	\node[main] (0) [below right of = 7] {$\mat{1}{0}{0}{1}$};
% 	\node[main] (15)[below left of = 0] {$\mat{0}{1}{1}{3}$};
% 	\node[main] (3)[below of = 15] {$\mat{1}{1}{1}{0}$};
% 	\draw[->] (0) --node[midway, above, sloped] {3} (15);
% 	\draw[->] (1) --node[midway, above, sloped] {1} (6);
% 	\draw[->] (2) --node[midway, above, sloped] {2} (15);
% 	\draw[->] (2) --node[midway, above, sloped] {3} (1);
% 	\draw[->] (3) --node[midway, above, sloped] {1} (12);
% 	\draw[->] (6) --node[midway, above, sloped] {1} (11);
% 	\draw[->] (6) --node[midway, above, sloped] {4} (7);
% 	\draw[->] (7) --node[midway, above, sloped] {1} (0);
% 	\draw[->] (11) --node[midway, above, sloped] {3} (7);
% 	\draw[->] (12) --node[midway, above, sloped] {1} (2);
% 	\draw[->] (15) --node[midway, above, sloped] {1} (3);
% \end{tikzpicture}

\begin{tikzpicture}[node distance={25mm}, thick, main/.style = {}]
	\node[main] (11)[] {$\mat{1}{2}{1}{1}$};
	\node[main] (6)[below right  of = 11] {$\mat{1}{1}{0}{1}$};
	\node[main] (7)[above right of = 11] {$\mat{1}{2}{1}{0}$};
	\node[main] (1)[right of = 6] {$\mat{0}{1}{1}{0}$};
	\node[main] (2)[right of = 1] {$\mat{1}{0}{1}{1}$};
	\node[main] (12)[right of = 2] {$\mat{1}{2}{0}{1}$};
	\node[main] (0) [right of = 7] {$\mat{1}{0}{0}{1}$};
	\node[main] (15)[right of = 0] {$\mat{0}{1}{1}{3}$};
	\node[main] (3)[right of = 15] {$\mat{1}{1}{1}{0}$};
	
	\draw[->] (0) --node[midway, above, sloped] {3} (15);
	\draw[->] (1) --node[midway, above, sloped] {1} (6);
	\draw[->] (2) --node[midway, above, sloped] {2} (15);
    \draw[->] (2) --node[midway, above, sloped] {3} (1);
	\draw[->] (3) --node[midway, above, sloped] {1} (12);
	\draw[->] (6) --node[midway, above, sloped] {1} (11);
	\draw[->] (6) --node[midway, above, sloped] {4} (7);
	\draw[->] (7) --node[midway, above, sloped] {1} (0);
	\draw[->] (11) --node[midway, above, sloped] {3} (7);
	\draw[->] (12) --node[midway, above, sloped] {1} (2);
	\draw[->] (15) --node[midway, above, sloped] {1} (3);
\end{tikzpicture}
	

%% file: graph4_64_after.tex
\begin{tikzpicture}[node distance={25mm}, thick, main/.style = {}]
	\node[main] (15)[] {$\mat{0}{0}{2}{5}$};
	\node[main] (5)[right of = 15] {$\mat{0}{0}{1}{11}$};
	\node[main] (12)[right of = 5] {$\mat{0}{0}{1}{4}$};
	\node[main] (8)[below of = 12] {$\mat{0}{0}{4}{1}$};
	\node[main] (6)[below of = 15] {$\mat{0}{0}{1}{6}$};
	\node[main] (13)[left of = 6] {$\mat{0}{0}{1}{13}$};
	\node[main] (11)[left of = 13] {$\mat{0}{0}{4}{3}$};
	\node[main] (17)[above of = 11] {$\mat{0}{0}{1}{12}$};
	
	\draw[->] (5) --node[midway, above, sloped] {1} (12);
	\draw[->] (6) --node[midway, above, sloped] {1} (15);
	\draw[->] (8) --node[midway, above, sloped] {2} (6);
	\draw[->] (11) --node[midway, above, sloped] {1} (13);
	\draw[->] (12) --node[midway, above, sloped] {1} (8);
	\draw[->] (13) --node[midway, above, sloped] {1} (6);
	\draw[->] (15) --node[midway, above, sloped] {1} (5);
	\draw[->] (15) --node[midway, above, sloped] {2} (17);
	\draw[->] (17) --node[midway, above, sloped] {1} (11);
\end{tikzpicture}

%% file: main.bbl
\begin{bibdiv}
\begin{biblist}

\bib{BaCr}{article}{
      author={Badkobeh, G.},
      author={Crochemore, M.},
       title={Finite-repetition threshold for infinite ternary words},
        date={2011},
     journal={Electronic Proceedings in Theoret. Comput. Sci.},
      volume={63},
      pages = {37\ndash 43},
}

\bib{Bar20}{misc}{
      author={Baranwal, A.~R.},
       title={Decision algorithms for {O}strowski-automatic sequences},
        date={2020},
         url={http://hdl.handle.net/10012/15845},
        note={master thesis, University of Waterloo},
}




\bib{BaSh19}{inproceedings}{
      author={Baranwal, A.~R.},
      author={Shallit, J.~O.},
       title={Critical exponent of infinite balanced words via the {Pell}
  number system},
        date={2019},
   booktitle={Combinatorics on words - 12th international conference, {WORDS}
  2019, proceedings},
      series={Lecture Notes in Computer Science},
      volume={11682},
   publisher={Springer},
       pages={80\ndash 92},
}

\bib{BeSe02}{incollection}{
      author={Berstel, J.},
      author={S\'e\'ebold, P.},
       title={Sturmian words},
        date={2002},
   booktitle={Algebraic combinatorics on words},
      editor={Lothaire, M.},
      series={Encyclopedia of Mathematics and Its Applications},
      volume={90},
   publisher={Cambridge University Press},
       pages={45\ndash 110},
}

\bib{Car07}{article}{
      author={Carpi, A.},
       title={On {Dejean's} conjecture over large alphabets},
        date={2007},
     journal={Theoret. Comput. Sci.},
      volume={385},
       pages={137\ndash 151},
}

\bib{CaLu2000}{article}{
      author={Carpi, A.},
      author={de~Luca, A.},
       title={Special factors, periodicity, and an application to {Sturmian}
  words},
        date={2000},
     journal={Acta Inf.},
      volume={36},
      number={12},
       pages={983\ndash 1006},
}

\bib{CuRa11}{article}{
      author={Currie, J.~D.},
      author={Rampersad, N.},
       title={A proof of {Dejean's} conjecture},
        date={2011},
     journal={Math. Comp.},
      volume={80},
       pages={1063\ndash 1070},
}

\bib{DaLe2002}{article}{
      author={Damanik, D.},
      author={Lenz, D.},
       title={The index of {S}turmian sequences},
        date={2002},
     journal={European J. Combinatorics},
      volume={23},
       pages={23\ndash 29},
}

\bib{Dej72}{article}{
      author={{Dejean}, F.},
       title={Sur un {th\'eor\`eme} de {Thue}},
        date={1972},
     journal={J. Combin. Theory. Ser. A},
      volume={13},
       pages={90\ndash 99},
}

\bib{De}{article}{
      author={Dekking, F. M.},
       title={On repetitions of blocks in binary sequences},
        date={1976},
     journal={J. Combin. Theory. Ser. A},
      volume={20},
      number={3},
       pages={292\ndash 299},
}

\bib{DolceDP21}{inproceedings}{
      author={Dolce, F.},
      author={Dvo{\v r}{\'{a}}kov{\'{a}}, L\!'.},
      author={Pelantov{\'{a}}, E.},
       title={Computation of critical exponent in balanced sequences},
        date={2021},
   booktitle={Combinatorics on words - 13th international conference, {WORDS}
  2021, proceedings},
      series={Lecture Notes in Computer Science},
      volume={12847},
   publisher={Springer},
       pages={78\ndash 90},
}

\bib{DDP21}{article}{
      author={Dolce, F.},
      author={Dvo{\v r}{\'{a}}kov{\'{a}}, L\!'.},
      author={Pelantov{\'{a}}, E.},
       title={On balanced sequences and their critical exponent},
        date={2021},
     journal={arXiv},
      volume={2108.07503},
         url={https://arxiv.org/abs/2108.07503},
}

\bib{Dur98}{article}{
      author={Durand, F.},
       title={A characterization of substitutive sequences using return words},
        date={1998},
     journal={Discrete Mathematics},
      volume={179},
       pages={89\ndash 101},
}

\bib{DvMePe20}{article}{
      author={Dvo{\v r}{\'{a}}kov{\'{a}}, L\!'.},
      author={Medkov{\'{a}}, K.},
      author={Pelantov{\'{a}}, E.},
       title={Complementary symmetric {R}ote sequences: the critical exponent
  and the recurrence function},
        date={2020},
     journal={Discret. Math. Theor. Comput. Sci.},
      volume={22},
      number={1},
}

\bib{DOPS2022}{article}{
      author={Dvo{\v r}{\'{a}}kov{\'{a}}, L\!'.},
      author={Opo\v censk\'a, D.},
      author={Pelantov{\'{a}}, E.},
      author={Shur, A.~M.},
       title={On minimal critical exponent of balanced sequences},
        date={2022},
     journal={Theoret. Comput. Sci.},
      volume={922},
       pages={158\ndash 169},
}

\bib{GoGrBrSh2019}{article}{
      author={Goulden, {I. P.}},
      author={Granville, A.},
      author={Richmond, {L. B.}},
      author={Shallit, J.},
       title={Natural exact covering systems and the reversion of the
  {M}{\"o}bius series},
        date={2019},
     journal={Ramanujan Journal},
      volume={50},
      number={1},
       pages={211\ndash 235},
}

\bib{Hof2011}{article}{
      author={Hoffman, J.~W.},
      author={Livingston, W.~R.},
      author={Ruiz, J.},
       title={A note on disjoint covering systems—variations on a 2002 aime
  problem},
        date={2011},
     journal={Mathematics Magazine, JSTOR},
      volume={84},
      number={3},
       pages={211\ndash 215},
}

\bib{Hubert00}{article}{
      author={Hubert, P.},
       title={Suites {\'{e}}quilibr{\'{e}}es},
        date={2000},
     journal={Theoret. Comput. Sci.},
      volume={242},
      number={1-2},
       pages={91\ndash 108},
}

\bib{JuPi02}{article}{
      author={Justin, J.},
      author={Pirillo, G.},
       title={Episturmian words and episturmian morphisms},
        date={2002},
     journal={Theoret. Comput. Sci.},
      volume={276},
      number={1-2},
       pages={281\ndash 313},
}

\bib{KrSh07}{article}{
      author={Krieger, D.},
      author={Shallit, J.~O.},
       title={Every real number greater than 1 is a critical exponent},
        date={2007},
     journal={Theoret. Comput. Sci.},
      volume={381},
      number={1-3},
       pages={177\ndash 182},
}

\bib{MoCu07}{article}{
      author={Mohammad-Noori, M.},
      author={Currie, J.~D.},
       title={Dejean's conjecture and {Sturmian} words},
        date={2007},
     journal={European J. Comb.},
      volume={28},
       pages={876\ndash 890},
}

\bib{MoHe40}{article}{
      author={Morse, M.},
      author={Hedlund, G.~A.},
       title={Symbolic dynamics {II}. {Sturmian} trajectories},
        date={1940},
     journal={American Journal of Mathematics},
      volume={62},
       pages={1\ndash 42},
}

\bib{Mou92}{article}{
      author={Moulin-Ollagnier, J.},
       title={Proof of {Dejean's} conjecture for alphabets with $5,6,7,8,9,10$
  and $11$ letters},
        date={1992},
     journal={Theoret. Comput. Sci.},
      volume={95},
      number = {2},
       pages={187\ndash 205},
}

\bib{Pan84c}{article}{
      author={Pansiot, J.-J.},
       title={A propos d'une conjecture de {F. Dejean} sur les
  {r\'ep\'etitions} dans les mots},
        date={1984},
     journal={Discr. Appl. Math.},
      volume={7},
      number={3},
       pages={297\ndash 311},
}

\bib{PoSch2002}{article}{
      author={Porubsk{\'y}, {\v S}.},
      author={Sch{\"o}nheim, J.},
       title={Covering systems of {P}aul {E}rd{\"o}s: Past, present and future,
  in {P}aul {E}rd{\"o}s and his mathematics},
        date={2002},
     journal={Bolyai Society Mathematical Studies},
      volume={1(11)},
       pages={581\ndash 627},
}

\bib{RSV19}{article}{
      author={Rampersad, N.},
      author={Shallit, J.~O.},
      author={Vandomme, {\'{E}}.},
       title={Critical exponents of infinite balanced words},
        date={2019},
     journal={Theoret. Comput. Sci.},
      volume={777},
       pages={454\ndash 463},
}

\bib{Rao11}{article}{
      author={Rao, M.},
       title={Last cases of {Dejean's} conjecture},
        date={2011},
     journal={Theoret. Comput. Sci.},
      volume={412},
      number={27},
       pages={3010\ndash 3018},
}

\bib{Schnabel2015}{article}{
      author={Schnabel, O.},
       title={On the reducibility of exact covering systems},
        date={2015},
     journal={INTEGERS},
      volume={15},
       pages={\#A34},
}

\bib{Sh}{article}{
      author={Shallit, J.},
       title={Simultaneous avoidance of large squares and fractional powers in
  infinite binary words},
        date={2004},
     journal={International Journal of Foundations of Computer Science},
      volume={15},
      number={02},
       pages={317\ndash 327},
}

\bib{Si85}{article}{
      author={Simpson, R.~J.},
       title={Regular coverings of the integers by arithmetic progressions},
        date={1985},
     journal={Acta Arithmetica},
      volume={45},
       pages={145\ndash 152},
}

\bib{ShTu}{inproceedings}{
      author={Tunev, I.~N.},
      author={Shur, A.~M.},
       title={On two stronger versions of Dejean's conjecture},
        date={2012},
   booktitle={Mathematical foundations of computer science 2012},
      editor={Rovan, Branislav},
      editor={Sassone, Vladimiro},
      editor={Widmayer, Peter},
   publisher={Springer Berlin Heidelberg},
     address={Berlin, Heidelberg},
       pages={800\ndash 812},
}

\bib{Vui01}{article}{
      author={Vuillon, L.},
       title={A characterization of {S}turmian words by return words},
        date={2001},
     journal={Eur. J. Comb.},
      volume={22},
      number={2},
       pages={263\ndash 275},
}

\bib{Ze72}{article}{
      author={Zeckendorf, E.},
       title={A generalized {F}ibonacci numeration},
        date={1972},
     journal={Fibonacci Quart.},
      volume={10},
       pages={365\ndash 372},
}

\bib{Znam1969}{article}{
      author={Znám, {\v S}.},
       title={On exactly covering system of arithmetic sequences},
        date={1969},
     journal={Math. Ann.},
      volume={180},
       pages={227\ndash 232},
}

\end{biblist}
\end{bibdiv}
